\documentclass[a4paper,12pt]{jpconf}

\usepackage{graphicx}
\usepackage{amsthm,amsmath,amsfonts,amssymb,mathrsfs}
\usepackage{lmodern}  
\usepackage{bbm} 

\newtheorem{definition}{Definition}[section]

\newtheorem{corollary}{Corollary}[section]
\newtheorem{proposition}{Proposition}[section]
\newtheorem{example}{Example}[section]
\newtheorem{remark}{Remark}[section]
\newtheorem{notation}{Notation}[section]

\newcommand{\df}[1]{\textbf{\textit{#1}}}  
\newcommand{\Tau}{\mathcal{T}} 
\newcommand{\E}{\mathbbmss{E}}
\newcommand{\U}{\mathbbmss{U}}
\newcommand{\PP}{\mathbb{P}} 
\newcommand{\NN}{\mathbb{N}} 
\newcommand{\Fb}{\mathbf{F}} 
\newcommand{\Rb}{\mathbf{R}} 
\newcommand{\Rs}{\mathscr{R}} 
\newcommand{\Ncal}{\mathcal{N}}  
\newcommand{\Acal}{\mathcal{A}}  
\newcommand{\Sup}{\mathrm{sup}}  
\newcommand{\SSE}[1][\U]{{\mathcal{SS}(#1)}_\E}  
\newcommand{\SPE}[1][\U]{{\mathcal{SP}(#1)}_\E}  
\newcommand{\STE}[1][X]{{\mathcal{S}\mathrm{Top}(#1)}_\E}  
\newcommand{\softsubseteq}{\tilde{\subseteq}}  
\newcommand{\softsupseteq}{\tilde{\supseteq}}  
\newcommand{\softequal}{\tilde{=}}  
\newcommand{\softin}{\tilde{\in}}  
\newcommand{\softnotin}{\tilde{\notin}}  
\newcommand{\nullsoftset}{(\tilde{\emptyset},\E)}  
\newcommand{\absolutesoftset}[1][\U]{(\tilde{#1},\E)}  
\newcommand{\softsetminus}{\widetilde{\setminus}}  
\newcommand{\softcup}{\tilde{\cup}}  
\newcommand{\softcap}{\tilde{\cap}}  
\newcommand{\softbigcup}{\widetilde{\bigcup}}  
\newcommand{\softbigcap}{\widetilde{\bigcap}}  
\newcommand{\softcl}[2][X]{s\textit{-}cl_#1\left( #2 \right)}  
\newcommand{\rl}[1][Y]{{}^{#1}\!}   
\newcommand{\graf}[1]{Gr(#1)}  
\newcommand{\apici}[1]{\lq\lq\textit{#1}\rq\rq}   

\begin{document}
\title{Soft $N$-Topological Spaces}

\author{Giorgio Nordo}
\address{MIFT - Dipartimento di Scienze Matematiche e Informatiche, scienze Fisiche e scienze della Terra,
Messina University, Messina, Italy}
\ead{giorgio.nordo@unime.it}

\begin{abstract}
Very recently, the idea of studying structures equipped with two or more soft topologies
has been considered by several researchers.
Soft bitopological spaces were introduced and studied, in 2014, by Ittanagi
as a soft counterpart of the notion of bitopological space.
and, independently, in 2015, by Naz, Shabir and Ali.
In 2017, Hassan too introduced the concept
of soft tritopological spaces and gave some first results.
The notion of $N$-topological space related to ordinary topological spaces
was instead introduced and studied, in 2011, by Tawfiq and Majeed.
In this paper we introduce the concept of Soft $N$-Topological Space
as generalization both of the concepts of Soft Topological Space and $N$-Topological Space
and we investigate such class of spaces and their basic properties with particular regard
to their subspaces, the parameterized families of crisp topologies generated by them
and some new separation axioms called $N$-wise soft $T_0$, $N$-wise soft $T_1$, and $N$-wise soft $T_2$.
\end{abstract}

\section{Introduction}
Inspired by a Pawlak's idea \cite{pawlak},
in 1999, Molodtsov \cite{molodtsov} initiated the novel concept of Soft Sets Theory as a new mathematical tool
and a completely different approach for dealing with uncertainties
while modelling problems in computer science, engineering physics, economics,
social sciences and medical sciences.
Molodtsov defines a soft set as a parameterized family of subsets of universe set where each element
is considered as a set of approximate elements of the soft set.

The absence of any restrictions on the approximate description in Soft Set Theory makes it
very convenient and easy to apply respect to other existing methods as Probability Theory and Fuzzy Set Theory.
In fact, we can define and use any kind of parametrization with the help of words, sentences,
real numbers, real functions, mappings, etc.

In the past few years, the fundamentals of soft set theory have been studied by many researchers.

Starting from 2002, Maji, Biswas and Roy \cite{maji2002, maji2003} studied the theory of soft sets initiated by
Molodstov, defining equality of two soft sets, subset and super set of a soft set,
complement of a soft set, null
soft set and absolute soft set with examples. Soft binary operations like AND, OR and the operations of union,
intersection were also defined.
In 2005, Pei and Miao \cite{pei} and Chen et al. \cite{chen} improved the work of Maji.

Further contributions to the Soft Sets Theory were given by Yang \cite{yang},
Ali et al. \cite{ali},
Fu \cite{fu},
Qin and Hong \cite{qin},
Sezgin and Atag\"{u}n \cite{sezgin},
Neog and Sut \cite{neog},
Ahmad and Kharal \cite{ahmad2009},
Babitha and Sunil \cite{babitha2010},
Ibrahim and Yosuf \cite{ibrahim},
Singh and Onyeozili \cite{singh},
Feng and Li \cite{feng},
Onyeozili and Gwary \cite{onyeozili}.

In the original formulation, every soft set is defined on a own subset of the common set of parameters
but, recently, Ma, Yang and Hu \cite{ma} proved that every soft set is equivalent
to a soft set related to the whole set of parameters.
This allow us to consider all the soft sets over the same parameter set and
simplify the definitions of all the relations and operations between them.
In particular, in 2014 \c{C}a\u{g}man \cite{cagman2014} improved and simplified
the definitions of operations on soft sets by using a single parameter set.

In 2011, Shabir and Naz \cite{shabir} introduced the concept of soft topological spaces,
also defining and investigating the notions of soft closed sets, soft closure,
soft neighborhood, soft subspace and some separation axioms.
Some other properties related to soft topology were studied by
\c{C}a\u{g}man, Karata\c{s} and Enginoglu in \cite{cagman2011}.
In the same year Hussain and Ahmad \cite{hussain} continued the study investigating
the properties of soft closed sets, soft neighbourhoods, sof interior, soft exterior
and soft boundary.
The notion of soft interior, soft neighbordhood and soft continuity were also
object of study by Zorlutuna, Akdag, Min and Atmaca in \cite{zorlutuna}.
Some other relations between such a notions was proved by Ahmad and Hussain in \cite{ahmad}.
The neighbourhood properties of a soft topological space were investigated in 2013
by Nazmul and Samanta \cite{nazmul}.

In \cite{min}, Min pointed out some errors contained in the Shabir--Naz paper and
investigated some properties of the separation axioms defined there.
The class of soft Hausdorff spaces was extensively studied by Varol and Ayg\"{u}n in \cite{varol2013}.

In 2012, Ayg\"{u}no\u{g}lu and Ayg\"{u}n \cite{aygunoglu}
defined and studied the notions of soft continuity and soft product topology.
Some years later, Zorlutuna and \c{C}aku \cite{zorlutuna2015}
gave some new characterizations of soft continuity, soft openness and soft closedness of soft mappings,
also generalyzing the Pasting Lemma to the soft topological spaces.
Soft first countable and soft second countable spaces were instead defined and studied by Rong in \cite{rong}.
Furthermore, the notion of soft continuity between soft topological spaces was independently introduced
and investigated by Hazra, Majumdar and Samanta in \cite{hazra}.

In 2013, Peyghan, Samadi and Tayebi \cite{peyghan} introduced the concept of soft connectedness
and soft Hausdorff space and investigated some related properties.
Soft connectedness was also studied in 2015 by Al-Khafaj \cite{al-khafaj} and Hussain \cite{hussain2015a}.
In the same year, Das and Samanta \cite{das,das2} introduced and extensively studied the soft metric spaces.

In 2014, the same three authors \cite{peyghan2014} defined also the notions of soft compactness
and countably soft compactness and obtain several results involving them and
some separation axioms introduced in the papers by Shabir-Naz \cite{shabir} and Min \cite{min}.
The notion of soft proximity was instead introduced and studied by Kandil, Tantawy,
El-Sheikh and Zakaria in \cite{kandil}.

In 2015, Hussain and Ahmad \cite{hussain2015b} redefine and explore several properties of
soft $T_i$ (with $i=0,1,2,3,4$) separation axioms and discuss some soft invariance properties
namely soft topological property and soft hereditary property.

In \cite{xie}, Xie introduced the concept of soft points and prove
that soft sets can be translated into soft point sets and then may conveniently deal
with soft sets as same as ordinary sets.

In the same year, Matedjdes \cite{matejdes2015} and, independently,
Shi and Pang \cite{shi} proved that the notion of soft topology is redundant,
i.e. that a soft topology in the sense of Shabir and Naz \cite{shabir}
can be interpreted as a classical (crisp) topology (by two different points of view).

In 2016, \"{O}zt\"{u}rk and Yolcu \cite{ozturk,ozturk2} introduced the notion of soft uniformity and
studied some properties of the soft uniform spaces
while Tantawy, El-Sheikh and Hamde \cite{tantawy} continued the study of soft $T_i$-spaces
(for $i=0,1,2,3,4,5$) also discussing the hereditary and topological properties for such spaces.

In 2017, Matejdes \cite{matejdes} studied various type of soft separation axioms and pointed out that
any soft topological space is homeomorphic to a crisp topological space defined on a cartesian product
and so that many soft topological notions and results can be directly derived from general topology.
For such a reason, Chiney and Samanta \cite{chiney} introduced a new definition of soft topology
by using elementary union and elementary intersection instead of the soft ones used
by Shabir and Naz and studied some basic properties of this new type of soft topological space.

Further contributions to the theory of soft sets and topology were added,
in 2012, by Varol, Shostak and Ayg\"{u}n \cite{varol},
by Janaki \cite{janaki},
in 2013 and 2014, by Wardowski \cite{wardowski}, Nazmul and Samanta \cite{nazmul2014},
by Georgiou, Megaritis and Petropoulos \cite{georgiou2013,georgiou2014},
in 2015 by Ulu\c{c}ay, \c{S}ahin, Olgun and Kili\c{c}man \cite{ulucay},
in 2016 by Wadkar, Bhardwaj, Mishra and Singh \cite{wadkar},
by Matejdes \cite{matejdes2016},
and by Fu and Fu \cite{fu2016},
in 2017 by Bdaiwi \cite{bdaiwi},
and in 2018 by Bayramov and Aras \cite{bayramov}.

In 1963, Kelly \cite{kelly} introduced the concept of bitopological space,
that is a structure with a pair of distincts topologies on the same set,
and studied their pairwise separation axioms.
In later years, many researchers (see, for example, \cite{kim,lal,lane,patty,pervin,reilly,singal})
have investigated bitopological spaces due to the richness of their structure
and potential for carrying out many generalization of classical topological results
in bitopological environment.

In 2003 Luay \cite{luay}, inspired by a previous work of Kov\'{a}r \cite{kovar},
gave a further generalization by introducing the notion of tripological space,
i.e. a set equipped with three different topologies on it.

Furthermore, in 2013, Mukundan introduced the notion of quad topological space and studied some sets
related to that space. In 2014, Tapi, Sharma and Deole \cite{tapi} defined and studied some
separation axioms in quad topological spaces.


In the last years, the idea of studying a soft set equipped with two or more soft topologies
has been considered by several researchers.
Soft bitopological spaces were introduced and studied, in 2014, by Ittanagi \cite{ittanagi}
as a soft counterpart of the notion of bitopological space.
and, independently, in 2015, by Naz, Shabir and Ali \cite{naz}
(under the slight different name of "bi-soft topological space").
In 2017, Hassan \cite{hassan} introduced also the concept of soft tritopological spaces
and gave some first results.
In the same year, Khattak et al. \cite{khattak} defined the notion of soft quad topological space
which involves four soft topologies and focused their study on some soft separation axioms in such a space.
In 2018, Khattak and some other researchers \cite{khattak2018} continued the investigation
studying the soft semi separation axioms in soft quad topological spaces.

In 2011, Tawfiq and Majeed \cite{tawfiq} and, independently, in 2012, Khan \cite{khan}
introduced the concept of $N$-topological space.

In the present paper we will present the notion of soft $N$-topological space as generalization both of the
concepts of soft topological space and $N$-topological space.

The rest of this work is organized as follows.
In the next section, concepts, notations and basic properties of soft sets and their operations are recalled.
In Section 3, the main notions of the theory of soft topology and some
fundamental properties are described and reviewed.
In Section 4, the new definition of $N$-soft topology is introduced and some basic properties
concerning -- in particular -- the subspace, the parameterized families of crisp topologies
generated by it and the new separation axioms called $N$-wise soft $T_0$, $N$-wise soft $T_1$,
and $N$-wise soft $T_2$ are investigated.
In the final section some concluding comments are summarized.

\section{Soft Sets}
In this section we present some basic definitions and results on soft sets and suitably exemplify them.
Terms and undefined concepts are used as in \cite{engelking}. 

\begin{definition}{\rm\cite{molodtsov}}
\label{def:softset}
Let $\U$ be an initial universe set and $\E$ be a nonempty set of parameters (or abstract attributes)
under consideration with respect to $\U$ and $A\subseteq \E$,
we say that a pair $(F,A)$ is a \df{soft set} over $\U$
if $F$ is a set-valued mapping $F: A \to \PP(\U)$
which maps every parameter $e \in A$ to a subset $F(e)$ of $\U$.
\end{definition}

In other words, a soft set is not a real (crisp) set
but a parameterized family $\left\{ F(e) \right\}_{e\in A}$ of subsets of the universe $\U$.
For every parameter $e \in A$, $F(e)$ may be considered as the set of \textit{$e$-approximate elements}
of the soft set $(F,A)$.

\begin{remark}
\label{rem:softsetsreducedtocrispsets}
Let us note that when the parameter set has only one element, i.e. when $\E = \{ \alpha \}$
any soft set $(F,A)$ is equivalent to the ordinary (crisp) set $F(\alpha)$.
\end{remark}

Although the Soft Sets Theory have had a great development in the past few years,
many researchers pointed out that some propositions, such as are 
generalization of De Morgan's Laws, Distributive Laws to soft sets are affected by errors that are
essentially due to some misunderstanding in the definition of the notions
of soft subset and soft intersections as given in \cite{maji2003}.

In 2010, Ma, Yang and Hu \cite{ma} proved that every soft set $(F,A)$ is equivalent
to the soft set $(F,\E)$ related to the whole set of parameters $\E$,
simply considering empty every approximations of parameters which are missing in $A$,
that is extending in a trivial way its set-valued mapping,
i.e. setting $F(e)=\emptyset$, for every $e \in \E \setminus A$.

For such a reason, in this paper we can consider all the soft sets over the same parameter set $\E$
as in \cite{chiney} and we will redefine all the basic operations and relations
between soft sets originally introduced in \cite{molodtsov,maji2002,maji2003} as in \cite{nazmul},
that is by considering the same parameter set.

\begin{remark}
\label{rem:softsetasgraph}
Another way to represent a soft set $(F,\E)$ is as the set of all the pairs $\left( e, F(e) \right)$
parameter-approximation (with $e \in \E$), i.e the graph $\graf{F} \subseteq E \times \PP(\U)$
of the set-valued mapping $F: \E \to \PP(\U)$.
In fact, in 2015, Matejdes \cite{matejdes2015} pointed out that there is a one-to-one correspondance
from the set $\Fb( \E,\PP(\U) )$ of all the set-valued mappings from $\E$ to $U$
onto the set $\Rb(\E,\U)$ of all binary relations from $\E$ to $\U$
(that is a bijective mapping $\Phi : \Fb( \E,\PP(\U) ) \to \Rb(\E,\U)$
defined, for any $F \in \Fb( \E,\PP(\U) )$ by $\Phi(F) = \Rs_F$
where $\Rs_F = \{ (e,u) \in \E\times \U : \, u \in F(e) \} = \graf{F}$
or, equivalently, for every $\Rs \in \Rb(\E,\U)$, by $\Phi^{-1}(\Rs) = F_{\Rs})$
where $F_{\Rs}: \E \to \PP(\U)$ is the set-valued mapping which maps
every $e \in \E$ in $F_{\Rs}(e)= \{ u \in \U : \, (e,u) \in \Rs \} = \Rs(e)$ )
and so that a soft set $(F,\E)$, which is substantially defined by $F \in \PP(\U)$,
bijectivly corresponds to the graph $\Phi(F) = \graf{F}$ of its set-valued mapping.
\end{remark}

\begin{example}
Here we present some examples of soft sets commonly used in the literature.
\label{ex:soft_set}
\begin{enumerate}
\item {\rm\cite{cagman2014}} Assume that there are six houses in the universe
$\U = \{h_l, h_2, h_3, h_4, h_5, h_6 \}$ under consideration, and that
$\E = \{e_1, e_2, e_3, e_4, e_5 \}$ is a set of decision parameters,
where the $e_i$ parameter (with $i=1,\ldots 5$) stands 
for \apici{expensive}, \apici{beautiful}, \apici{wooden}, \apici{cheap}
and \apici{in green surroundings} respectively.
Consider the subset of parameters $A = \{e_1, e_3, e_4 \}$ and define
the set-valued mapping \apici{house with some characteristic} $F : A \to \PP(\U)$ by setting
$F(e_1) = \{ h_2, h_4 \}$, $F(e_3) = \U$ and $F(e_4) = \{ h_1, h_3, h_5 \}$.
Then the pair $(F,A)$ is a soft set which espresses the fact that the houses $h_2$ and $h_4$ are expensive,
all the houses are wooden and the house $h_1, h_3, h_5$ are cheap
and it can also be represented by the graph of $F$ consisting of the following family of approximations
$\graf{F}= \left\{ (e_1,\{ h_2, h_4 \}), (e_3, \U), (e_4,\{ h_1, h_3, h_5 \}) \right\}$.

\item {\rm\cite{molodtsov}} Let $(X,\Tau)$ be a topological space on a nonempty set $X$.
If, for every $x \in X$, we consider the family $\Ncal^o_x = \{ N \in \Tau : x \in \Tau \}$
of all open neighbourhoods of $x$ and define a set-valued mapping $T : \ X \to \PP(X)$
by setting $T(x)=\Ncal^o_x$ (for any $x \in X$), then the pair $(X, T)$ is a soft set over the universe $X$.

\item {\rm\cite{molodtsov}}
Let $A$ be a fuzzy set over a universe $\U$ and let $\mu_a : \U \to [0,1]$ be
its membership function (see \cite{zadeh}).
For every $\alpha \in [0,1]$, consider the $\alpha$-level set
$A^{\ge \alpha} = \{ x\in \U : \mu_A(x) \ge \alpha \}$ and define a set-valued mapping
$F : [0,1] \to \PP(\U)$ by setting $F(\alpha) = A^{\ge \alpha}$ for any $\alpha \in [0,1]$.
Since, for every $x \in A$, we know that $\mu_A(x) = \Sup \{\alpha \in [0,1], \, x \in A^{\ge \alpha} \}
= \Sup \{\alpha \in [0,1], \, x \in F(\alpha) \} $,
ii follows that  every Zadeh's fuzzy set $A$ may be considered as a special case of soft set
and more precisely the soft set $(F, [0, 1])$ over the same universe $\U$.
\end{enumerate}
\end{example}

\begin{definition}{\rm\cite{zorlutuna}}
\label{def:setofsoftsets}
The set of all the soft sets over a universe $\U$ with respect to a set of parameters $\E$
will be denoted by $\SSE$.
\end{definition}

\begin{definition}{\rm\cite{nazmul}}
\label{def:softsubset}
Let $(F,\E),(G,\E) \in \SSE$ be two soft sets over a common universe $\U$
and a common set of parameters $\E$,
we say that $(F,\E)$ is a \df{soft subset} of $(G,\E)$ and we write
$(F,\E) \softsubseteq (G,\E)$
if $F(e)\subseteq G(e)$ for every $e \in \E$.
\end{definition}

\begin{definition}{\rm\cite{nazmul}}
\label{def:softsuperset}
Let $(F,\E),(G,\E) \in \SSE$ be two soft sets over a common universe $\U$
and a common set of parameters $\E$,
we say that $(F,\E)$ is a \df{soft super set} of $(G,\E)$ and we write
$(F,\E) \softsupseteq (G,\E)$
if $(G,\E)$ is a soft subset of $(F,\E)$, i.e. if $(G,\E) \softsubseteq (F,\E)$
\end{definition}

\begin{definition}{\rm\cite{nazmul}}
\label{def:softequal}
Let $(F,\E),(G,\E) \in \SSE$ be two soft sets over a common universe $\U$, we say that
$(F,\E)$ and $(G,\E)$ are \df{soft equal} and we write $(F,\E) \softequal (G,\E)$
if $(F,\E) \softsubseteq (G,\E)$ and $(G,\E) \softsubseteq (F,\E)$.
\end{definition}


\begin{definition}{\rm\cite{nazmul}}
\label{def:nullsoftset}
A soft set $(F,\E)$ over a universe $\U$ is said to be \df{null soft set}
and it is denoted by $\nullsoftset$ if $F(e) = \emptyset$ for every $e \in \E$.
\end{definition}

\begin{definition}{\rm\cite{nazmul}}
\label{def:absolutesoftset}
A soft set $(F,\E) \in \SSE$ over a universe $\U$ is said to be a \df{absolute soft set}
and it is denoted by $\absolutesoftset$
if $F(e) = \U$ for every $e \in \E$.
\end{definition}

Clearly, for every soft set $(F,\E) \in \SSE$, we have
$\nullsoftset \softsubseteq (F,\E) \softsubseteq \absolutesoftset$.


\begin{definition}
\label{def:constantsoftset}
Let $(F,\E) \in \SSE$ be a soft set over a universe $\U$
and $V$ be a nonempty subset of $U$,
the \df{constant soft set} of $V$, denoted by $(\tilde{V},\E)$)
(or, sometimes, by $\tilde{V}$), is the soft set $(\underbar{V},\E)$,
where $\underbar{V}: \E \to \PP(\U)$ is the constant set-valued mapping
defined by $\underbar{V}(e) = V$, for every $e \in \E$.
\end{definition}

\begin{proposition}{\rm\cite{cagman2014}}
\label{pro:propertiessubsets}
For every triplet $(F,\E), (G,\E), (H,\E) \in \SSE$ of soft sets, we have that
$(F,\E) \softsubseteq (G,\E)$ and $(G,\E) \softsubseteq (H,\E)$
imply $(F,\E) \softsubseteq (H,\E)$.
\end{proposition}


\begin{definition}{\rm\cite{nazmul}}
\label{def:softcomplement}
Let $(F,\E) \in \SSE$ be a soft set over a universe $\U$, the \df{soft complement}
(or more exactly the \textit{soft relative complement}) of $(F,\E)$,
denoted with $(F,\E)^\complement$, is the soft set $\left( F^\complement, E \right)$
where $F^\complement : \E \to \PP(\U)$ is the set-valued mapping
defined by $F^\complement(e) = F(e)^\complement = \U \setminus F(e)$, for every $e \in \E$.
\end{definition}

It is routine to show that the soft complement of the null soft set is soft equal to the absolute soft set, i.e.
$\nullsoftset^\complement \, \softequal \, \absolutesoftset$,
that the soft complement of the absolute soft set is soft equal to the null soft set, i.e.
$\absolutesoftset^\complement \, \softequal \, \nullsoftset$,
and that the soft complement of the soft complement of any soft set $(F,\E)$ is soft equal
to the soft set itself, i.e.
$\left( (F,\E)^\complement \right)^\complement \softequal \, (F,\E)$.

\begin{definition}{\rm\cite{nazmul}}
\label{def:softdifference}
Let $(F,\E),(G,\E) \in \SSE$ be two soft sets over a common universe $\U$,
the \df{soft difference} of $(F,\E)$ and $(G,\E)$,
denoted by $(F,\E) \softsetminus (G,\E)$, is the soft set $\left( F \setminus G, E \right)$
where $F \setminus G : \E \to \PP(\U)$ is the set-valued mapping
defined by $(F \setminus G)(e) = F(e) \setminus G(e)$, for every $e \in \E$.
\end{definition}

Clearly, for every soft set $(F,\E) \in \SSE$, it results
$(F,\E)^\complement \, \softequal \, \absolutesoftset \softsetminus (F,\E)$.

\begin{definition}{\rm\cite{nazmul}}
\label{def:softunion}
Let $(F,\E), (G,\E) \in \SSE$ be two soft sets over a universe $\U$,
the \df{soft union} of $(F,\E)$ and $(G,\E)$, denoted with $(F,\E) \softcup (G,\E)$,
is the soft set $\left( F \cup G, E \right)$
where $F \cup G: \E \to \PP(\U)$ is the set-valued mapping
defined by $(F \cup G)(e) = F(e) \cup G(e)$, for every $e \in \E$.
\end{definition}

\begin{definition}{\rm\cite{nazmul}}
\label{def:softintersection}
Let $(F,\E), (G,\E) \in \SSE$ be two soft sets over a universe $\U$,
the \df{soft intersection} of $(F,\E)$ and $(G,\E)$, denoted with $(F,\E) \softcap (G,\E)$,
is the soft set $\left( F \cap G, E \right)$
where $F \cap G: \E \to \PP(\U)$ is the set-valued mapping
defined by $(F \cap G)(e) = F(e) \cap G(e)$, for every $e \in \E$.
\end{definition}

\begin{definition}{\rm\cite{al-khafaj}}
\label{def:softdisjunct}
Two soft sets $(F,\E)$ and $(G,\E)$ over a common universe $\U$
are said to be \df{soft disjoint} if their soft intersection is the soft null set,
i.e. if $(F,\E) \softcap (G,\E) \, \softequal \, \nullsoftset$.
\end{definition}

\begin{remark}
\label{rem:softdisjunction}
It is a simple matter to verify that two soft sets $(F,\E)$ and $(G,\E)$ are soft disjoint
according to Definition \ref{def:softdisjunct}, if and only if every their corresponding
approximations is disjoint, that is $F(e)\cap G(e) = \emptyset$, for every $e \in \E$.
\\
Let us note that in some paper (see, for example, \cite{khattak}) the last equivalent expression
is assumed as definition.
\end{remark}


The soft operators of union, intersection and complement satisfy relations
similar to those of (crisp) set theory, such as the well-known commutative, associative,
distributive, exclusion, contradiction and De Morgan's Laws.

\begin{proposition}{\rm\cite{cagman2014}}
\label{pro:propertiesunionandintersection}
For every soft set $(F,\E) \in \SSE$, we have:
\begin{enumerate}
\item $(F,\E) \softcup (F,\E) \softequal (F,\E)$
\item $(F,\E) \softcup \nullsoftset \softequal (F,\E)$
\item $(F,\E) \softcup \absolutesoftset \softequal \absolutesoftset)$
\item $(F,\E) \softcap (F,\E) \softequal (F,\E)$
\item $(F,\E) \softcap \nullsoftset \softequal \nullsoftset$
\item $(F,\E) \softcap \absolutesoftset \softequal (F,\E)$
\end{enumerate}
\end{proposition}

\begin{proposition}{\rm\cite{cagman2014}}
\label{pro:commutativesoftsets}
For every pair $(F,\E), (G,\E) \in \SSE$ of soft sets, we have:
\begin{enumerate}
\item $(F,\E) \softcup (G,\E) \softequal (G,\E) \softcup (F,\E)$
\item $(F,\E) \softcap (G,\E) \softequal (G,\E) \softcap (F,\E)$
\end{enumerate}
\end{proposition}

\begin{proposition}{\rm\cite{cagman2014}}
\label{pro:associativeanddistributivesoftsets}
For every triplet $(F,\E), (G,\E), (H,\E) \in \SSE$ of soft sets, we have:
\begin{enumerate}
\item $(F,\E) \softcap \left( (G,\E) \softcap (H,\E) \right) \softequal
  \left( (F,\E) \softcap (G,\E) \right) \softcap (H,\E) $
\item $(F,\E) \softcup \left( (G,\E) \softcup (H,\E) \right) \softequal
  \left( (F,\E) \softcup (G,\E) \right) \softcup (H,\E) $
\item $(F,\E) \softcap \left( (G,\E) \softcup (H,\E) \right) \softequal
  \left( (F,\E) \softcap (G,\E) \right) \softcup \left( (F,\E) \softcap (H,\E) \right)$
\item $(F,\E) \softcup \left( (G,\E) \softcap (H,\E) \right) \softequal
  \left( (F,\E) \softcup (G,\E) \right) \softcap \left( (F,\E) \softcup (H,\E) \right)$
\end{enumerate}
\end{proposition}


\begin{proposition}{\rm\cite{neog}}
\label{pro:softexclusionandcontradiction}
For every soft set $(F,\E) \in \SSE$, we have:
\begin{enumerate}
\item $(F,\E) \softcup (F,\E)^\complement \, \softequal \, \absolutesoftset$,
\item $(F,\E) \softcap (F,\E)^\complement \, \softequal \, \nullsoftset$
\end{enumerate}
\end{proposition}

\begin{proposition}{\rm\cite{zorlutuna}}
\label{pro:softsubset_and_softoperators}
Let $(F,\E), (G,\E) \in \SSE$ be two soft sets over a universe $\U$, then:
\begin{enumerate}
\item $(F,\E) \softsubseteq (G,\E)$ iff $(F,\E) \softcap (G,\E) \softequal (F,\E)$
\item $(F,\E) \softsubseteq (G,\E)$ iff $(F,\E) \softcup (G,\E) \softequal (G,\E)$
\end{enumerate}
\end{proposition}

\begin{proposition}{\rm\cite{shabir}}
\label{pro:demorganslaws}
Let $(F,\E), (G,\E) \in \SSE$ be two soft sets over a universe $\U$, then:
\begin{enumerate}
\item $\left ( (F,\E) \softcup (G,\E) \right)^\complement
  \softequal \, (F,\E)^\complement \, \softcap \, (G,\E)^\complement$
\item $\left ( (F,\E) \softcap (G,\E) \right)^\complement
  \softequal \, (F,\E)^\complement \, \softcup \, (G,\E)^\complement$
\end{enumerate}
\end{proposition}

\begin{proposition}{\rm\cite{shabir}}
\label{pro:differencesoftsets}
Let $(F,\E), (G,\E) \in \SSE$ be two soft sets over a universe $\U$, we have that
$(F,\E) \softsetminus (G,\E) \softequal (F,\E) \softcap (G,\E)^\complement$.
\end{proposition}

\begin{proposition}{\rm\cite{cagman2014}}
\label{pro:propertiessoftdifference}
For every soft set $(F,\E) \in \SSE$, we have:
\begin{enumerate}
\item $(F,\E) \softsetminus (F,\E) \softequal \nullsoftset$
\item $(F,\E) \softsetminus \nullsoftset \softequal (F,\E)$
\item $\nullsoftset \softsetminus (F,\E) \softequal \nullsoftset$
\item $(F,\E) \softsetminus \absolutesoftset \softequal \nullsoftset$
\end{enumerate}
\end{proposition}


The notions of soft union and intersection admits a obvious generalization to a family
with any number of soft sets.

\begin{definition}{\rm\cite{nazmul}}
\label{def:generalizedsoftunion}
Let $\left\{(F_i,\E) \right\}_{i\in I} \subseteq \SSE$ be a nonempty subfamily
of soft sets over a universe $\U$,
the (generalized) \df{soft union} of $\left\{(F_i,\E) \right\}_{i\in I}$,
denoted with $\softbigcup_{i \in I} (F_i,\E) $,
is defined by $\left(\bigcup_{i \in I} F_i, \E \right)$
where $\bigcup_{i \in I} F_i: \E \to \PP(\U)$ is the set-valued mapping
defined by $\left(\bigcup_{i \in I} F_i\right)(e) = \bigcup_{i \in I} F_i(e)$, for every $e \in \E$.
\end{definition}

\begin{definition}{\rm\cite{nazmul}}
\label{def:generalizedsoftintersection}
Let $\left\{(F_i,\E) \right\}_{i\in I} \subseteq \SSE$ be a nonempty subfamily
of soft sets over a universe $\U$,
the (generalized) \df{soft intersection} of $\left\{(F_i,\E) \right\}_{i\in I}$,
denoted with $\softbigcap_{i \in I} (F_i,\E) $,
is defined by $\left(\bigcap_{i \in I} F_i, E \right)$
where $\bigcap_{i \in I} F_i: \E \to \PP(\U)$ is the set-valued mapping
defined by $\left(\bigcap_{i \in I} F_i\right)(e) = \bigcap_{i \in I} F_i(e)$, for every $e \in \E$.
\end{definition}

\begin{proposition}
\label{pro:intersectionanduniongeneralized}
Let $\left\{(F_i,\E) \right\}_{i\in I} \subseteq \SSE$ be a nonempty subfamily
of soft sets over a universe $\U$, the for every $i \in I$, we have that:
$\softbigcap_{i \in I} (F_i,\E) \, \softsubseteq \, (F_i,\E)
\, \softsubseteq \, \softbigcup_{i \in I} (F_i,\E) $.
\end{proposition}


Propositions \ref{pro:associativeanddistributivesoftsets}(iii)-(iv) and \ref{pro:demorganslaws}
can be easily extended to arbitrary union and arbitrary intersection.

\begin{proposition}
\label{pro:generalizeddistributive}
Let respectively $(F,\E) \in \SSE$ be a soft set
and $\left\{(G_i,\E) \right\}_{i\in I} \subseteq \SSE$ be a nonempty subfamily
of soft sets over a common universe $\U$, we have:
\begin{enumerate}
\item $(F,\E) \softcap \left( \softbigcup_{i \in I} (G_i,\E) \right) \softequal \,
\softbigcup_{i \in I} \left( (F,\E) \softcap (G_i,\E) \right)$
\item $(F,\E) \softcup \left( \softbigcap_{i \in I} (G_i,\E) \right) \softequal \,
\softbigcap_{i \in I} \left( (F,\E) \softcup (G_i,\E) \right)$
\end{enumerate}
\end{proposition}

\begin{proposition}
\label{pro:generalizeddemorganlaws}
Let $\left\{(F_i,\E) \right\}_{i\in I} \subseteq \SSE$ be a nonempty subfamily
of soft sets over a universe $\U$, it results:
\begin{enumerate}
\item $\left( \softbigcup_{i \in I} (F_i,\E) \right)^\complement \softequal \,
\softbigcap_{i \in I} (F_i,\E)^\complement $
\item $\left( \softbigcap_{i \in I} (F_i,\E) \right)^\complement \softequal \,
\softbigcup_{i \in I} (F_i,\E)^\complement $
\end{enumerate}
\end{proposition}

Despite the name, soft sets are not real sets since they defined by means of a set-valued mapping.
For such a reason, their original definition lacked the concept of point.
In 2015, Xie \cite{xie} introduced the notion of soft point and study some relationships between
soft points and soft sets, finding in particular that soft sets
can be converted into ordinary sets of soft points so that we may conveniently deal with soft relations,
soft operations and so on.

\begin{definition}{\rm\cite{xie}}
\label{def:softpoint}
A soft set $(F,\E) \in \SSE$ over a universe $\U$ is said to be a \df{soft point} over $U$
if it has only one non-empty approximation and it is a singleton,
i.e. if there exists some parameter $\alpha \in E$
and an element $p \in \U$ such that
$F(\alpha) = \{ p \}$ and $F(e)=\emptyset$ for every $e \in E \setminus \{ \alpha \}$.
Such a soft point is usually denoted with $(p_\alpha, \E)$.
The singleton $\{ p \}$ is called the \textit{support set} of the soft point
and $\alpha$ is called the \textit{expressive parameter} of $(p_\alpha, \E)$.
\end{definition}

\begin{remark}
Let us observe that the soft point notation $(p_\alpha, \E)$ maintains
and makes immediately recognizable both the salient information,
that is the value of the parameter and that of the point itself.
Every reference to the set-valued mapping from which it derives is completely superfluous
since it has only one non-empty value $\{p\}$ corresponding to the parameter $\alpha$.
\\
In other words, a soft point $(p_\alpha, \E)$ is a soft set corresponding to
the set-valued mapping $p_\alpha : \E \to \mathbb(U)$ that, for any $e \in \E$, is defined by
$$ p_\alpha(e) =
\left\{
\begin{array}{ll}
\{ p \} & \mbox{ if } e = \alpha \\
\emptyset & \mbox{ if } e \in \E \setminus \{\alpha \} \\
\end{array}
\right.
$$
\end{remark}

\begin{example}
\label{ex:soft_point}
Let $\U = \{h_l, h_2, h_3, h_4, h_5, h_6 \}$ the universe set
and $\E = \{e_1, e_2, e_3, e_4, e_5 \}$ the set of decision parameters
as in the Example \ref{ex:soft_set}(iii).
Then, if we consider the fixed value $p= h_5 \in \U$, the fixed value $\alpha = e_2 \in \E$
and the set-valued mapping $p_\alpha : \E \to \PP(\U)$  defined by setting
$p_\alpha(e_1) = \emptyset$, $p_\alpha(e_2) = \{ p \} = \{ h_5 \} $, $p_\alpha(e_3) = \emptyset$,
$p_\alpha(e_4) = \emptyset$ and $p_\alpha(e_5) = \emptyset$,
the soft set $(p_\alpha, \E)$ (or more precisely $\left( {h_5}_{e_2} , \E \right)$) is a soft point.
\end{example}

\begin{definition}{\rm\cite{xie}}
\label{def:setofsoftpoints}
The set of all the soft points over a universe $\U$ with respect to a set of parameters $\E$
will be denoted by $\SPE$.
\end{definition}

\begin{definition}{\rm\cite{xie}}
\label{def:softpointsoftbelongstosoftset}
Let $(p_\alpha, \E) \in \SPE$ and $(F,\E) \in \SSE$
respectively be a soft point and a softset over a common universe $\U$.
We say that \df{the soft point $(p_\alpha, \E)$ soft belongs to the soft set $(F,\E)$}
and we write $(p_\alpha, \E) \softin (F,\E)$, if the soft point is a soft subset of the soft set,
i.e. if $(p_\alpha, \E) \softsubseteq (F,\E)$
and hence if $p \in F(\alpha)$.
\end{definition}

\begin{definition}{\rm\cite{das}}
\label{def:equalitysoftpoints}
Let $(p_\alpha, \E), (q_\beta, \E) \in \SPE$ be two soft points over a common universe $\U$,
we say that $(p_\alpha, \E)$ and $(q_\beta, \E)$ are \df{soft equal},
and we write $(p_\alpha, \E) \softequal (q_\beta, \E)$,
if they are equals as soft sets and hence if $p = q$ and $\alpha = \beta$.
\end{definition}

\begin{definition}{\rm\cite{das}}
\label{def:distinctssoftpoints}
We say that two soft points $(p_\alpha, \E)$ and $(q_\beta, \E)$ are \df{soft distincts},
and we write $(p_\alpha, \E) \tilde{\ne} (q_\beta, \E)$,
if and only if $p\ne q$ or $\alpha \ne \beta$.
\end{definition}

In the special case in which the soft points $p \in F(\alpha)$ and $q \in F(\alpha)$
are defined respect to the same expressive parameter $\alpha$ it is obvious that
they are soft equal if and only if $p = q$
and soft distincts if and only if $p\ne q$.

\begin{remark}
\label{rem:noteonsoftdistinctstpoints}
In some papers (see, for example \cite{khattak} and \cite{khattak2018}),
using a different notation, two soft points $(p_\alpha, \E)$ and $(q_\beta, \E)$
satisfying Definition \ref{def:distinctssoftpoints}
are said "disjoint" but for uniformity of language, it seems to us
more appropriated to define them as distinct.
\end{remark}

\begin{proposition}{\rm\cite{das}}
\label{pro:softsetasunionofsoftpoint}
Any soft set $(F,\E) \in \SSE$ over a universe $\U$ can be represented
as soft union of all its soft points, i.e.
$$ (F,\E) = \softbigcup \left\{ (p_\alpha, \E) : \, (p_\alpha, \E) \softin (F,\E) \right\} .$$
\end{proposition}

A different notion of membership, used in particular, for defining soft separation axioms
is given in \cite{shabir} by Shabir and Naz.

\begin{definition}{\rm\cite{shabir}}
\label{def:pointbelongstosoftset}
Let $(F,\E) \in \SSE$ be a soft set over a universe $\U$ and $p \in \U$.
We say that \df{the point $p$ belongs to the soft set $(F,\E)$} and we write that
$p \in (F,\U)$ if $p \in F(e)$, for every $e \in \E$.
\end{definition}

By the above Definition, it is immediately clear that
$p \notin (F,\E)$ if there exists some $\alpha \in E$ such that $p \notin F(\alpha)$.

\begin{remark}
\label{rem::pointbelongstosoftset}
Let us note that a (ordinary) point $p \in \U$ belongs to a soft set $(F,\E) \in \SSE$ if and only if
every corresponding soft point $(p_e, \E)$ relative to the same support set $\{ p \}$
and any expressive parameter $e \in \E$ soft belongs to the soft set, i.e. that
$p \in (F,\E)$ if and only if $(p_e, \E) \softin (F,\E)$ for every $e \in \E$.
Thus, the soft membership $(p_\alpha, \E) \softin (F,\E)$ is a weaker condition than
to the ordinary membership $p \in (F,\E)$.
\end{remark}

\begin{definition}{\rm\cite{hussain}}
\label{def:subsoftset}
Let $(F,\E) \in \SSE$ be a soft set over a universe $\U$
and $V$ be a nonempty subset of $\U$,
the \df{sub soft set} of $(F,\E)$ over $V$, is the soft set $(\rl[V] F,\E)$,
where $\rl[V] F: \E \to \PP(\U)$ is the set-valued mapping
defined by $\rl[V] F(e) = F(e) \cap V$, for every $e \in \E$.
\end{definition}

\begin{remark}
\label{rem:subsoftset}
Using Definitions \ref{def:constantsoftset} and \ref{def:softintersection},
it is a trivial matter to verify that
a sub soft set of $(F,\E)$ over $V$ can also be expressed as
$(\rl[V] F,\E) = (F,\E) \softcap (\tilde{V}, \E)$.
\end{remark}

\begin{example}{\rm\cite{hussain}}
\label{ex:sub_soft_set}
Let $\U = \{h_l, h_2, h_3, h_4, h_5, h_6 \}$ the universe set,
$\E = \{e_1, e_2, e_3, e_4, e_5 \}$ the set of decision parameters
as in the Example \ref{ex:soft_set}(i)
and consider the soft set $(F,\E)$ defined by
$$
\begin{array}{lll}
\hspace{8mm} F(e_1)= \{ h_2 , h_4 \},
\hspace{10mm} & F(e_2) = \{ h_1, h_3, h_4, h_6 \},
\hspace{6mm} & F(e_3) = \{ h_2, h_3, h_6 \}, \\
\hspace{8mm} F(e_4)= \{ h_1, h_3, h_5, h_6 \},
& F(e_5) = \{ h_2, h_3, h_4, h_6 \} . &
\end{array}
$$
Now, if we consider the subset $V = \{ h_1, h_2, h_3, h_4 \}$ of the universe $U$, the sub soft set
$(\rl[V] F,\E)$ of $(F,\E)$ over $V$ results defined by
$$
\begin{array}{lll}
\hspace{2mm} \rl[V] F(e_1)= \{ h_2 , h_4 \},
\hspace{12mm} & \rl[V] F(e_2) = \{ h_1, h_3, h_4 \},
\hspace{10mm} & \rl[V] F(e_3) = \{ h_2, h_3 \}, \\
\hspace{2mm} \rl[V] F(e_4)= \{ h_1, h_3 \},
& \rl[V] F(e_5) = \{ h_2, h_3, h_4 \} . &
\end{array}
$$
\end{example}

\section{Soft Topological Spaces}
The notion of soft topological spaces as topological spaces defined over a initial universe
with a fixed set of parameters was introduced in 2011 by Shabir and Naz \cite{shabir}.

\begin{definition}{\rm\cite{shabir}}
\label{def:softtopology}
Let $X$ be an initial universe set, $\E$ be a nonempty set of parameters with respect to $X$
and $\Tau \subseteq \SSE[X]$ be a family of soft sets over $X$, we say that
$\Tau$ is a \df{soft topology} on $X$ with respect to $\E$ if the following four conditions are satisfied:
\begin{enumerate}
\item the null soft set belongs to $\Tau$, i.e. $\nullsoftset \in \Tau$
\item the absolute soft set belongs to $\Tau$, i.e. $\absolutesoftset[X] \in \Tau$
\item the soft intersection of any two soft sets of $\Tau$ belongs to $\Tau$, i.e.
for every $(F,\E), (G,\E) \in \Tau$ then $(F,\E) \softcap (G,\E) \in \Tau$.
\item the soft union of any subfamily of soft sets in $\Tau$ belongs to $\Tau$, i.e.
for every $\left\{(F_i,\E) \right\}_{i\in I} \subseteq \Tau$ then
$\softbigcup_{i \in I} (F_i,\E) \in \Tau$
\end{enumerate}
The triplet $(X, \Tau, \E)$ is called a \df{soft topological space} over $X$ with respect to $\E$.
\\
In some case, when it is necessary to better specify the universal set and the set of parameters,
the topology will be denoted by $\Tau(X,\E)$.
\end{definition}

It is a trivial fact to verify that condition \textit{(iii)} of Definition \ref{def:softtopology}
is equivalent to state that the soft intersection of any finite number of soft sets of $\Tau$ belongs to $\Tau$,
i.e. that for every $\left\{(F_i,\E) \right\}_{i=1,\ldots n} \subseteq \Tau$ then
$\softbigcap_{i=1}^n (F_i,\E) \in \Tau$.

\begin{definition}{\rm\cite{shabir}}
\label{def:softopenset}
Let $(X,\Tau,\E)$ be a soft topological space over $X$ with respect to $\E$,
then the members of $\Tau$ are said to be \df{soft open set} in $X$.
\end{definition}

\begin{example}
\label{ex:soft_topology}
For the sake of clarity, here are some examples of soft set families that form a soft topology or not.
\begin{enumerate}
\item {\rm\cite{shabir}}
Let $X$ be an initial universe set, $\E$ be the set of parameters.
If we consider the family of soft sets defined by
$\Tau_b = \left\{ \nullsoftset, \absolutesoftset[X] \right\}$,
then $\Tau_b$ is a soft topology.
We say that $\Tau_b$ is the \df{soft indiscrete topology} (or soft trivial topology) on $X$
and we call $(X,\Tau_b,\E)$ the \df{soft indiscrete topological space}
(or soft trivial topological space) over $X$.

\item {\rm\cite{shabir}}
Let $X$ be an initial universe set, $\E$ be the set of parameters.
If we consider the family $\Tau_d = \SSE[X]$ of all soft sets over $X$,
then $\Tau_d$ is a soft topology.
We say that $\Tau_d$ is the \df{soft discrete topology} on $X$
and we call $(X,\Tau_d,\E)$ the \df{soft discrete topological space} over $X$.

\item Let $X= \{ h_1, h_2, h_3 \}$ be the universe set and $\E = \{ e_1, e_2 \}$ be the set of parameters.
Consider the family of soft sets $\Tau = \left\{ \nullsoftset, \absolutesoftset[X],
(F_1,\E), (F_2,\E), (F_3,\E), (F_4,\E) \right\}$ where the soft sets $(F_i,\E)$ (with $i=1,\ldots 4$)
over $X$ are defined by setting:
$$
\begin{array}{ll}
F_1(e_1)= \{ h_2 \}, \qquad & F_1(e_2) = \{ h_1 \}, \\
F_2(e_1)= \{ h_1, h_2 \},  & F_2(e_2) = \{ h_1 \}, \\
F_3(e_1)= \{ h_1, h_2 \},  & F_3(e_2) = \{ h_1, h_3 \}, \\
F_4(e_1)= \{ h_1, h_2 \},  & F_4(e_2) = X .
\end{array}
$$
Since, it is a simple routine to verify that all non-trivial unions and intersections
of the members of the family $\Tau$ still belong to $\Tau$ and, more exactly that
$$
\begin{array}{lll}
(F_1,\E) \softcup (F_2,\E) = (F_2,\E), \quad &
(F_1,\E) \softcup (F_3,\E) = (F_3,\E), \quad &
(F_1,\E) \softcup (F_4,\E) = (F_4,\E) \\
(F_2,\E) \softcup (F_3,\E) = (F_3,\E), &
(F_2,\E) \softcup (F_4,\E) = (F_4,\E), &
(F_3,\E) \softcup (F_4,\E) = (F_4,\E), \\
(F_1,\E) \softcap (F_2,\E) = (F_1,\E), &
(F_1,\E) \softcap (F_3,\E) = (F_1,\E), &
(F_1,\E) \softcap (F_4,\E) = (F_1,\E), \\
(F_2,\E) \softcap (F_3,\E) = (F_2,\E), &
(F_2,\E) \softcap (F_4,\E) = (F_2,\E), &
(F_3,\E) \softcap (F_4,\E) = (F_3,\E) .
\end{array}
$$
by Definition \ref{def:softtopology} it follows $\Tau$ is a soft topology on $X$
and hence that $(X,\Tau,\E)$ is a (finite) soft topological space over $X$.

\item {\rm\cite{hussain}}
Let $X= \{ h_1, h_2, h_3 \}$ be the universe set and $\E = \{ e_1, e_2 \}$ be the set of parameters.
Consider the family of soft sets $\Acal = \left\{ \nullsoftset, \absolutesoftset[X],
(F_1,\E), (F_2,\E), (F_3,\E), (F_4,\E) \right\}$
where the soft sets $(F_i,\E)$ (with $i=1,\ldots 4$) over $X$ are defined by setting:
$$
\begin{array}{ll}
F_1(e_1)= \{ h_2 \}, \qquad & F_1(e_2) = \{ h_1 \}, \\
F_2(e_1)= \{ h_2, h_3 \},  & F_2(e_2) = \{ h_1 ,h_2 \}, \\
F_3(e_1)= \{ h_1, h_2 \},  & F_3(e_2) = \{ h_1, h_2 \}, \\
F_4(e_1)= \{ h_2 \},  & F_4(e_2) =  \{ h_1, h_3 \} .
\end{array}
$$
Since, it can be easily verified that the soft union $(F_2,\E) \softcup (F_3,\E)$ is
the soft set $(H,\E)$ defined by $H(e_1) = X$ and $H(e_2)=\{ h_1 ,h_2 \}$ and so that
$(H,\E) \notin \Acal$ it follows that $\Acal$ is not a soft topology on $X$.
\end{enumerate}
\end{example}

\begin{definition}{\rm\cite{shabir}}
\label{def:softclosedset}
Let $(X,\Tau,\E)$ be a soft topological space over $X$ and be $(F,\E)$ be a soft set over $X$.
We say that $(F,\E)$ is \df{soft closed set} in $X$ if its complement $(F,\E)^\complement$
is a soft open set, i.e. if $(F^\complement,\E) \in \Tau$.
\end{definition}

\begin{notation}
The family of all soft closed sets of a soft topological space $(X,\Tau,\E)$ over $X$ with respect to $\E$
will be denoted by $\sigma$,
or more precisely with $\sigma(X,\E)$ when it is necessary to specify
the universal set $X$ and the set of parameters $\E$.
\end{notation}

Using Proposition \ref{pro:demorganslaws} and Proposition \ref{pro:generalizeddemorganlaws}
with Definition \ref{def:softtopology}, the following statement is immediately proved.

\begin{proposition}{\rm\cite{shabir}}
\label{pro:propertiesofsoftclosedsets}
Let $\sigma$ be the family of soft closed sets of a soft topological space $(X,\Tau,\E)$, the following hold:
\begin{enumerate}
\item the null soft set is a soft closed set, i.e. $\nullsoftset \in \sigma$
\item the absolute soft set is a soft closed set, i.e. $\absolutesoftset[X] \in \sigma$
\item the soft union of any two soft closed sets is still a soft closed set, i.e.
for every $(C,\E), (D,\E) \in \sigma$ then $(C,\E) \softcup (D,\E) \in \sigma$.
\item the soft intersection of any subfamily of soft closed sets is still a soft closed set, i.e.
for every $\left\{(C_i,\E) \right\}_{i\in I} \subseteq \sigma$ then
$\softbigcap_{i \in I} (C_i,\E) \in \sigma$
\end{enumerate}
\end{proposition}

As a consequence of Proposition \ref{pro:propertiesofsoftclosedsets}, the family $\sigma(X,\E)$
is sometimes called the soft closed topology over $X$.

The following is a trivial generalization of a proposition given in \cite{shabir}.

\begin{proposition}
\label{pro:intersectionofsofttopologies}
Let $N \in \NN$ be an integer number greater than $0$
and $\left\{ (X,\Tau_i,\E) \right\}_{i=1,\ldots N}$ be a finite family
of $N$ soft topological spaces over the same universe $X$
with respect to $\E$,
then the (crisp) intersection $\bigcap_{i=1}^N \Tau_i$ of all its soft topologies $\Tau_i$
(with $i=1,\ldots N$)
is still a soft topology on $X$
and so $\left( X, \bigcap_{i=1}^N \Tau_i, \E \right)$ is a soft topological space over $X$.
\end{proposition}

\begin{remark}
\label{rem:unionofsofttopologies}
The following example shows that, unlike intersections,
the (crisp) union of (two) soft topologies is not necessarily a soft topology.
\end{remark}

\begin{example}
\label{ex:union_of_two_soft_topologies}
Let $X= \{ h_1, h_2, h_3 \}$ be the universe set and $\E = \{ e_1, e_2 \}$ be the set of parameters.
Consider the family of soft sets $\Tau_1 = \left\{ \nullsoftset, \absolutesoftset[X],
(F_1,\E), (F_2,\E), (F_3,\E), (F_4,\E)  \right\}$
where the soft sets $(F_i,\E)$ (with $i=1,\ldots 4$) over $X$ are defined by setting:
$$
\begin{array}{ll}
F_1(e_1)= \{ h_2 \}, \qquad & F_1(e_2) = \{ h_1 \}, \\
F_2(e_1)= \{ h_1, h_2 \},  & F_2(e_2) = \{ h_1 \}, \\
F_3(e_1)= \{ h_1, h_2 \},  & F_3(e_2) = \{ h_1, h_3 \}, \\
F_4(e_1)= \{ h_1, h_2 \},  & F_4(e_2) = X .
\end{array}
$$
and $\Tau_2 = \left\{ \nullsoftset, \absolutesoftset[X],
(G_1,\E), (G_2,\E), (G_3,\E) \right\}$
where the soft sets $(G_i,\E)$ (with $i=1,\ldots 3$) over $X$ are defined by setting:
$$
\begin{array}{ll}
G_1(e_1)= \{ h_2 \}, \qquad & G_1(e_2) = \{ h_1 \}, \\
G_2(e_1)= \{ h_2 \},  & G_2(e_2) = \{ h_1, h_3 \}, \\
G_3(e_1)= \{ h_2, h_3 \},  & G_3(e_2) = \{ h_1, h_3 \} .
\end{array}
$$
It is a simple routine to verify that $\Tau_1$ and $\Tau_2$ verify
the axioms from Definition \ref{def:softtopology}
(in particular $\Tau_1$ coincides with the soft topology $\Tau$ of the Example \ref{ex:soft_topology}(iii))
and so that they are soft topologies on $X$.

It is worth noting that, according to Proposition \ref{pro:intersectionofsofttopologies},
the crisp intersection intersection of the two soft topologies $\Tau_1$ and $\Tau_2$, that is
$$\Tau_1 \cap \Tau_2 = \left\{ \nullsoftset, \absolutesoftset[X], (F_1,\E) \right\}$$
is indeed a soft topology over the same universe $X$.

However, if we consider the union
$$\Acal = \Tau_1 \cup \Tau_2 = \left\{ \nullsoftset, \absolutesoftset[X],
(F_1,\E), (F_2,\E), (F_3,\E), (F_4,\E), (G_2,\E), (G_3,\E) \right\}$$
we can observe that $(F_3, \E) \softcup (G_2, \E)$ is the soft set $(H,\E)$
defined by $H(e_1) = X$ and $H(e_2) = \{ h_1, h_3 \}$ and so that $(H,\E) \notin \Acal$.
This proves that the union $\Acal = \Tau_1 \cup \Tau_2$ of the two soft topologies $\Tau_1$ and $\Tau_2$
is a family of soft sets which is not a soft topology on $X$.
\end{example}

Although the union of soft topologies is not in general a soft topology,
it will be useful to give the following definition.

\begin{definition}
\label{def:supsofttopologies}
Let $\left\{ (X,\Tau_i,\E) \right\}_{i=1,\ldots N}$ be a finite family
of soft topological spaces over the same universe $X$ with respect to $\E$.
The \df{supremum soft topology} of all soft topologies $\Tau_i$ (with $i=1,\ldots N$),
denoted by $\bigvee_{i=1}^N \Tau_i$,
is the smallest soft topology on $X$ containing the (crisp) union $\bigcup_{i=1}^N \Tau_i$.
The soft topological space $(\left( X, \bigvee_{i=1}^N \Tau_i, \E \right)$
is called the \df{supremum soft topological space}.
\end{definition}

\begin{example}
\label{ex:supsofttopologies}
Let $X= \{ h_1, h_2, h_3 \}$ be the universe set, $\E = \{ e_1, e_2 \}$ be the set of parameters
and consider the two soft topologies $\Tau_1 = \left\{ \nullsoftset, \absolutesoftset[X],
(F_1,\E), (F_2,\E), (F_3,\E), (F_4,\E)  \right\}$
and $\Tau_2 = \left\{ \nullsoftset, \absolutesoftset[X], (G_1,\E), (G_2,\E), (G_3,\E) \right\}$
defined in the Example \ref{ex:union_of_two_soft_topologies}.

The supremum soft topology $\Tau_1 \vee \Tau_2$ is the smallest soft topology over $X$
which contain the crisp union $\Tau_1 \cup \Tau_2$.
Thus, after noticing that
$(F_1,\E) \softcup (G_2,\E) = (F_4,\E) \softcup (G_2,\E) = (G_2,\E)$,
$(F_2,\E) \softcup (G_2,\E) = (F_3,\E) \softcup (G_2,\E) = (F_3,\E)$,
$(F_1,\E) \softcup (G_3,\E) = (G_3,\E)$
and that
$(F_1,\E) \softcap (G_2,\E) = (F_1,\E) \softcap (G_3,\E) = (F_2,\E) \softcap (G_3,\E) =  (F_1,\E)$,
$(F_2,\E) \softcap (G_2,\E) = (F_3,\E) \softcap (G_2,\E) = $
$(G_2,\E) = (F_3,\E) \softcap (G_3,\E)
= (F_4,\E) \softcap (G_2,\E) = (F_4,\E) \softcap (G_3,\E) = (G_2,\E)$,
it follows that:
$$\Tau_1 \vee \Tau_2
 = \left\{ \nullsoftset, \absolutesoftset[X], (F_1,\E), (F_2,\E), (F_3,\E), (F_4,\E),
  (G_2,\E), (G_3,\E), (H,\E) \right\}$$
where the unique new soft open set
$(H,\E) = (F_2,\E) \softcup (G_3,\E) = (F_3,\E) \softcup (G_3,\E) = (F_4,\E) \softcup (G_3,\E) $
is defined by $H(e_1) = X$ and $H(e_2) = \{ h_1, h_3 \}$.
\end{example}

\begin{notation}
\label{not:infofsofttopologies}
The unique soft topology $\bigcap_{i=1}^N \Tau_i$ of Proposition \ref{pro:intersectionofsofttopologies}
is evidently the \df{infimum} of all the soft topologies $\Tau_i$ (with $i=1,\ldots N$)
and it is generally denoted by $\bigwedge_{i=1}^N \Tau_i$.
\end{notation}

Thus, the set $\STE$ of all soft topologies over a universe $X$ with respect to the set of parameters $\E$,
equipped with the two binary operations $\wedge$ and $\vee$, that is $(\STE, \wedge, \vee)$ is a lattice.


\begin{proposition}{\rm\cite{shabir}}
\label{pro:softtopologytotopologies}
Let $(X,\Tau,\E)$ be a soft topological space over $X$. Then, for every $e \in \E$,
the family $\Tau_e = \left\{ F(e) : \, (F,\E) \in \Tau \right\}$ is a (crisp) topology on $X$.
\end{proposition}

\begin{remark}
\label{rem:familyoftopologies_no_induces_soft_topology}
Proposition \ref{pro:softtopologytotopologies} shows that every soft topological space $(X,\Tau,\E)$
gives a parameterized family $\left\{ (X,\Tau_e) \right\}_{e \in \E}$ of ordinary topological spaces.
By way, as shown in the following counterexample, 
the converse does not hold, i.e. a family $\Acal \subseteq \SSE[X]$ of soft sets
is not in general a soft topology
even if every family $\Acal_e = \left\{ F(e) : \, (F,\E) \in \Acal \right\}$
of sets corresponding to each parameter $e \in \E$ defines a topology.
\end{remark}

\begin{example}
\label{ex:familyoftopologies_no_induces_soft_topology}
Let $X= \{ h_1, h_2, h_3 \}$ be the universe set and $\E = \{ e_1, e_2 \}$ be the set of parameters.
Consider the family $\Acal = \left\{ \nullsoftset, \absolutesoftset[X],
(F_1,\E), (F_2,\E), (F_3,\E), (F_4,\E) \right\}$ of soft sets
$(F_i,\E)$ (with $i=1,\ldots 4$) over $X$ defined by:
$$
\begin{array}{ll}
F_1(e_1)= \{ h_2 \}, \qquad & F_1(e_2) = \{ h_1 \}, \\
F_2(e_1)= \{ h_2, h_3 \},  & F_2(e_2) = \{ h_1 ,h_2 \}, \\
F_3(e_1)= \{ h_1, h_2 \},  & F_3(e_2) = \{ h_1, h_2 \}, \\
F_4(e_1)= \{ h_2 \},  & F_4(e_2) =  \{ h_1, h_3 \} .
\end{array}
$$
that, in Example \ref{ex:soft_topology}(iv), we have already proved
to do not form a soft topology over $X$.
Also, the crisp families corresponding to the parameters $e_1$ and $e_2$ respectively, i.e.
$$
\Acal_{e_1} = \left\{ \emptyset, X, \{h_2 \}, \{h_1, h_2\},  \{h_2, h_3\} \right\}
$$
and
$$
\Acal_{e_1} = \left\{ \emptyset, X, \{h_1 \}, \{h_1, h_2\},  \{h_1, h_3\} \right\}
$$
are topologies on the set $X$.
\end{example}

Let us also observe that in \cite{hazra} the point of view is reversed and the previous Proposition
is taken as a different definition for a topology of soft subsets over $X$.
So, Proposition \ref{pro:softtopologytotopologies} says that every soft topological space
in the sense of Shabir--Naz \cite{shabir} is also a soft topological space
in the sense of Hazra--Mujumdar--Samanta \cite{hazra}.


\begin{definition}{\rm\cite{zorlutuna}}
\label{pro:softneighbourhoodofasoftpoint}
Let $(X, \Tau, \E)$ be a soft topological space, $(N,\E) \in \SSE[X]$ be a soft set
and $(x_\alpha, \E) \in \SPE[X]$ be a soft point over a common universe $X$.
We say that $(N,\E)$ is a \df{soft neighbourhood} of the soft point $(x_\alpha, \E)$
if there is some soft open set soft containing the soft point and soft contained in the soft set,
that is if there exists some soft open set $(A,\E) \in \Tau$
such that $(x_\alpha, \E) \softin (A,\E) \softsubseteq (N,\E)$.
\end{definition}

\begin{notation}
\label{not:familyofsoftneighbourhoodofasoftpoint}
The family of all soft neighbourhoods of a soft point $(x_\alpha, \E) \in \SPE[X]$
in a soft topological space $(X, \Tau, \E)$ will be denoted by
$\Ncal_{(x_\alpha, \E)}$
(or more precisely with $\Ncal^\Tau_{(x_\alpha, \E)}$ if it is necessary to specify the topology).
\end{notation}

\begin{proposition}{\rm\cite{zorlutuna}}
\label{pro:propertiesofsoftneighbourhoodofasoftpoint}
The family $\Ncal_{(x_\alpha, \E)}$ of all the soft neighbourhoods
of a soft point $(x_\alpha, \E) \in \SPE[X]$
in a soft topological space $(X, \Tau, \E)$ over $X$ satisfies the following four properties:
\begin{enumerate}
\item for any $(N,\E) \in \Ncal_{(x_\alpha, \E)}$, $(x_\alpha, \E) \softin (N,\E)$
\item for any $(M,\E), (N,\E) \in \Ncal_{(x_\alpha, \E)}$,
$(M,\E) \softcap (N,\E) \in \Ncal_{(x_\alpha, \E)}$
\item for any $(N,\E) \in \Ncal_{(x_\alpha, \E)}$ and every soft set $(F,\E) \in \SSE[X]$
such that $(N,\E) \softsubseteq (F,\E)$ then $(F,\E) \in \Ncal_{(x_\alpha, \E)}$
\item for any $(N,\E) \in \Ncal_{(x_\alpha, \E)}$
there exists some $(M,\E) \in \Ncal_{(x_\alpha, \E)}$ such that
for every soft point $(y_\beta, \E) \softin (M,\E)$
then $(N,\E) \in \Ncal_{(y_\beta, \E)}$
\end{enumerate}
\end{proposition}

\begin{proposition}{\rm\cite{georgiou2013}}
\label{pro:characterizationofsoftopensetsbysoftneighbourhoods}
Let $(X, \Tau, \E)$ be a soft topological space over the universe $X$.
Then a soft set $(F,\E) \in \SSE[X]$ is a soft open set if and only if
for every soft point $(x_\alpha, \E) \in (F,\E)$ there exists a soft open neighbourhood
$(N,\E) \in \Ncal_{(x_\alpha, \E)}$ such that
$(x_\alpha, \E) \softin (N,\E) \softsubseteq (F,\E)$.
\end{proposition}

\begin{corollary}
\label{cor:characterizationofsoftopensets}
Let $(X, \Tau, \E)$ be a soft topological space over $X$.
Then a soft set $(F,\E) \in \SSE[X]$ is a soft open set if and only if
it is a soft open neighbourhood of every its soft point $(x_\alpha, \E) \in (F,\E)$.
\end{corollary}


\begin{definition}{\rm\cite{shabir}}
\label{def:softclosure}
Let $(X,\Tau,\E)$ be a soft topological space over $X$ and $(F,\E)$ be a soft set over $X$.
Then the \df{soft closure} of the soft set $(F,\E)$,
denoted by $\softcl{(F,\E)}$,
is the soft intersection of all soft closed set over $X$ soft containing $(F,\E)$, that is
$$\softcl{(F,\E)} \softequal \, \softbigcap \left\{ (C,\E) \in \sigma(X,\E) :
\, (F,\E) \softsubseteq (C,\E) \right\}$$
\end{definition}

Clearly, $\softcl{(F,\E)}$ is the smallest soft closed set over $X$ which soft contains $(F,\E)$.

\begin{proposition}{\rm\cite{shabir}}
\label{pro:propertiesofsoftclosure}
Let $(X,\Tau,\E)$ be a soft topological space over $X$, and $(F,\E)$ be a soft set over $X$.
Then the following hold:
\begin{enumerate}
\item $\softcl{ \nullsoftset } \softequal \, \nullsoftset$
\item $\softcl{ \absolutesoftset[X] } \softequal \, \absolutesoftset[X]$
\item $(F,\E) \, \softsubseteq \, \softcl{(F,\E)}$
\item $(F,\E)$ is a soft closed set over $X$ if and only if $\softcl{(F,\E)} \softequal \, (F,\E)$
\item $\softcl{ \softcl{(F,\E)} } \softequal \, \softcl{(F,\E)}$
\end{enumerate}
\end{proposition}

\begin{proposition}{\rm\cite{shabir}}
\label{pro:sofsoftclosureofoperators}
Let $(X,\Tau,\E)$ be a soft topological space and
$(F,\E), (G,\E) \in \SSE[X]$ be two soft sets over a common universe $X$.
Then the following hold:
\begin{enumerate}
\item $(F,\E) \softsubseteq (G,\E)$ implies $\softcl{(F,\E)} \softsubseteq \, \softcl{(G,\E)}$
\item $\softcl{ (F,\E) \softcup (G,\E) } \softequal \, \softcl{(F,\E)} \softcup \, \softcl{(G,\E)}$
\item $\softcl{ (F,\E) \softcap (G,\E) } \softsubseteq \, \softcl{(F,\E)} \softcap \, \softcl{(G,\E)}$
\end{enumerate}
\end{proposition}


Having in mind the Definition \ref{def:subsoftset} we can recall the following proposition.

\begin{proposition}{\rm\cite{hussain}}
\label{pro:softrelativetopology}
Let $(X,\Tau,\E)$ be a soft topological space over $X$, and $Y$ be a nonempty subset of $X$, then
the family $\Tau_Y$ of all sub soft sets of $\Tau$ over $Y$, i.e.
$$\Tau_Y = \left\{ (\rl F,\E) : \, (F,\E) \in \Tau \right\}$$
is a soft topology on $Y$.
\end{proposition}

\begin{definition}{\rm\cite{hussain}}
\label{def:softrelativetopology}
Let $(X,\Tau,\E)$ be a soft topological space over $X$, and $Y$ be a nonempty subsets of $X$,
the soft topology $\Tau_Y = \left\{ (\rl F,\E) : \, (F,\E) \in \Tau \right\}$
is said to be the \df{soft relative topology} of $\Tau$ on $Y$
and $(Y,\Tau_Y,\E)$ is called a \df{soft topological subspace} of $(X,\Tau,\E)$ on $Y$.
\end{definition}

\begin{example}
\label{ex:softrelativetopology}
Let $X= \{ h_1, h_2, h_3 \}$ be the universe set and $\E = \{ e_1, e_2 \}$ be the set of parameters.
Consider the family of soft sets $\Tau = \left\{ \nullsoftset, \absolutesoftset[X],
(F_1,\E), (F_2,\E), (F_3,\E), (F_4,\E) \right\}$
where the soft sets $(F_i,\E)$ (with $i=1,\ldots 4$) over $X$ are defined by setting:
$$
\begin{array}{ll}
F_1(e_1)= \{ h_2 \}, \qquad & F_1(e_2) = \{ h_1 \}, \\
F_2(e_1)= \{ h_1, h_2 \},  & F_2(e_2) = \{ h_1 \}, \\
F_3(e_1)= \{ h_1, h_2 \},  & F_3(e_2) = \{ h_1, h_3 \}, \\
F_4(e_1)= \{ h_1, h_2 \},  & F_4(e_2) = X .
\end{array}
$$
that in Example \ref{ex:soft_topology}(iii) we have already proved to be a soft topology on $X$.
Now, if we consider the subset $Y = \{ h_1, h_2 \}$ of $X$,
the sub soft sets $(\rl F_i,\E)$ (with $i=1,\ldots 4$) of the soft open set $(F_i,\E)$ over $Y$
results to be defined by:
$$
\begin{array}{ll}
\rl F_1(e_1)= \{ h_2 \}, \qquad & \rl F_1(e_2) = \{ h_1 \}, \\
\rl F_2(e_1)= Y,  & \rl F_2(e_2) = \{ h_1 \}, \\
\rl F_3(e_1)= Y,  & \rl F_3(e_2) = \{ h_1 \}, \\
\rl F_4(e_1)= Y,  & \rl F_4(e_2) = Y .
\end{array}
$$
and so, being $(\rl F_2,\E) \softequal (\rl F_3,\E)$ and $(\rl F_4,\E) = \absolutesoftset[Y]$,
the soft relative topology $\Tau_Y$ of $\Tau$ on $Y$ is:
$$
\Tau_Y = \left\{ \nullsoftset, \absolutesoftset[Y], (\rl F_1,\E), (\rl F_2,\E) \right\} .
$$
\end{example}


\begin{definition}{\rm\cite{hussain2015b}}
\label{def:softT0}
A soft topological space $(X,\Tau,\E)$ over $X$ is called a \df{soft $T_0$-space}
if for every pair of distinct soft points $(x_\alpha, \E), (y_\beta, \E) \in \SPE[X]$
there exists a soft open set which soft contains exactly one of these soft points, i.e.
there is some $(A,\E) \in \Tau$ such that $(x_\alpha, \E) \softin (A,\E)$
and $(y_\beta, \E) \softnotin (A,\E)$,
or there is some $(B,\E) \in \Tau$ such that $(y_\beta, \E) \softin (B,\E)$
and $(x_\alpha, \E) \softnotin (B,\E)$.
\end{definition}

\begin{definition}{\rm\cite{hussain2015b}}
\label{def:softT1}
A soft topological space $(X,\Tau,\E)$ over $X$ is called a \df{soft $T_1$-space}
if for every pair of distinct soft points $(x_\alpha, \E), (y_\beta, \E) \in \SPE[X]$
there exists a soft open set $(A,\E) \in \Tau$ which soft contains one soft point but not the other one, that is
$(x_\alpha, \E) \softin (A,\E)$ and $(y_\beta, \E) \softnotin (A,\E)$.
\end{definition}

\begin{remark}
\label{rem:differenceindefiningseparationaxioms}
Let us note that some other authors
(see, for example, \cite{georgiou2013}, \cite{nazmul}, \cite{peyghan}, \cite{shabir} and \cite{tantawy})
defined the soft separation axioms using ordinary points and not soft points,
that is referring to the Definition \ref{def:pointbelongstosoftset}
given by Xie in \cite{xie}
instead of the Definition \ref{def:softpointsoftbelongstosoftset}, which
-- in our opinion -- appears more proper and correct.
\end{remark}

\begin{proposition}{\rm\cite{hussain2015b}}
\label{pro:characterizationsoftT1space}
A soft topological space $(X,\Tau,\E)$ over $X$ is a soft $T_1$-space
if and only if every its soft point $(x_\alpha, \E) \in \SPE[X]$
is a soft closed set.
\end{proposition}

\begin{definition}{\rm\cite{hussain2015b}}
\label{def:softT2}
A soft topological space $(X,\Tau,\E)$ over $X$ is called a \df{soft $T_2$-space}
(or a \df{soft Hausdorff space})
if for every pair of distinct soft points $(x_\alpha, \E), (y_\beta, \E) \in \SPE[X]$
there exist two soft open sets $(A,\E), (B,\E) \in \Tau$ which are soft disjoints
and soft contain the two soft points respectively, that is
$(x_\alpha, \E) \softin (A,\E)$, $(y_\beta, \E) \softin (B,\E)$
and $(A,\E) \softcap (B,\E) = \nullsoftset$.
\end{definition}

\begin{remark}
\label{rem:implicationsbetweenT0T1T2spaces}
Evidently, every soft $T_2$-space is a soft $T_1$-space
and every soft $T_1$-space is a soft $T_0$-space.
\end{remark}

In 2015, Matejdes \cite{matejdes2015}, after having noted that
every soft set $(F,\E)$ bijectively corresponds to the graph $\graf{F}$ of its set-valued mapping
(see Remark \ref{rem:softsetasgraph}),
proved that $(X,\Tau,\E)$ is a soft topological space if and only if
the set $\Tau_{\E\times X} = \{ \graf{F} : (F,\E) \in \Tau \}$
of the graphs corresponding to the set-valued mappings of all the soft open sets of $\Tau$
forms an ordinary topology on the cartesian product $\E\times X$,
i.e. if $\left( \E\times X, \Tau_{\E\times X} \right)$ is a (crisp) topological space.
\\
Matejdes also proved that a soft set $(F,\E)$ is a soft open set in a soft topological space $(X,\Tau,\E)$
if and only if the graph $\graf{F}$ corresponding to the set-valued mapping is an open set
in the topological space $\left( \E\times X, \Tau_{\E\times X} \right)$ defined on the cartesian product.
Hence, every topological notion can be introduced for a soft topological space $(X,\Tau,\E)$
by direct reformulation of that notion in the topological space $\left( \E\times X, \Tau_{\E\times X} \right)$.


\section{Soft $N$-Topological Spaces}
Very recently, the idea of studying structures equipped with two or more soft topologies
has been considered by several researchers.
Soft bitopological spaces were introduced and studied, in 2014, by Ittanagi \cite{ittanagi}
as a soft counterpart of the notion of bitopological space.
and, independently, in 2015, by Naz, Shabir and Ali \cite{naz}
(under the slight different name of "bi-soft topological space").
In 2017, Hassan \cite{hassan} introduced also the concept
of soft tritopological spaces and gave some first results,
while Khattak et al. \cite{khattak} defined the notion of soft quad topological space
whose study continued in \cite{khattak2018}.

The concept of $N$-topological space related to ordinary topological spaces
was introduced and studied, in 2011, by Tawfiq and Majeed \cite{tawfiq}
and, independently, in 2012, by Khan \cite{khan}.

In this section we initiate the study of soft $N$-Topological Spaces
as a natural soft counterpart of the notion above,
in order to extends and generalizes the results on soft bitopological and soft tritopological spaces.

\begin{definition}
\label{def:softntopologicalspace}
Let $X$ be an initial universe set, $\E$ be a nonempty set of parameters with respect to $X$,
$N \in \NN$ be an integer number greater than $0$
and $(X,\Tau_i, \E)$ (with $i=1,\ldots N$) be $N$ different soft topological spaces over the same universe $X$,
then the soft set $X$ equipped with all these topologies will be said \df{soft $N$-topological space} over $X$
and will be denoted by $\left( X, \Tau_i, N, \E \right)$.
For any $i=1,\ldots N$, a member of $\Tau_i$ is said to be a $\Tau_i$ soft open set.
A complement of a $\Tau_i$ soft open set is called a $\Tau_i$ soft closed set.
\end{definition}

\begin{remark}
\label{rem:softnspacegeneralization}
Evidently Definition \ref{def:softntopologicalspace} generalizes several already studied
classes of soft topological spaces since
for $N=2$ we have the notion of soft bitopological spaces
introduced and studied in 2014, by Ittanagi \cite{ittanagi}
and, independently, in 2015, by Naz, Shabir and Ali \cite{naz};
for $N=3$ we obtain the definition of soft tripological spaces defined in \cite{hassan} by Hassan, and
for $N=4$ we have the notion of soft quad topological space introduced
by Khattak et al. in \cite{khattak}.
\end{remark}

\begin{definition}
\label{def:softnopenset}
Let $\left( X, \Tau_i, N, \E \right)$ be a soft $N$-topological space over $X$.
A soft set $(F,\E)$ over $X$ is said to be a \df{soft $N$-open set}
if it is a $\Tau_j$ soft open set for some $j = 1,\ldots N$, i.e.
if there exists some $j \in\ \{1,\ldots N\}$ such that  $(F,\E) \in \Tau_j$.
\end{definition}

\begin{remark}
\label{rem:alternativedefinitionofsoftnopenset}
Let us note that Definition \ref{def:softnopenset} is equivalent to say that
$(F,\E) \in \bigcup_{i=1}^N \Tau_i$ where the union operator is the usual set-union
and not a soft union as defined in Definition \ref{def:generalizedsoftunion}.

Furthermore, recalling Remark \ref{rem:unionofsofttopologies}, it is clear that
$\bigcup_{i=1}^N \Tau_i$ is not, in general, a soft topology.
\end{remark}


\begin{definition}
\label{def:softnclosedset}
Let $\left( X, \Tau_i, N, \E \right)$ be a soft $N$-topological space over $X$.
A soft set $(G,\E)$ over $X$ is said to be a \df{soft $N$-closed set} if
its soft complement is $(G,\E)^\complement$ is a soft $N$-open set.
\end{definition}

Evidently, a soft set $(G,\E)$ is a soft $N$-closed set if it is at least a
$\Tau_j$ soft closed set for some $j=1,\ldots N$.

\begin{example}
\label{ex:softntopologicalspace}
Let $X= \{ h_1, h_2, h_3, h_4, h_5, h_6 \}$ be the universe set
and $\E = \{ e_1, e_2, e_3 \}$ be the set of parameters.
Consider the following $N=4$ soft topologies over $X$:
$$
\begin{array}{ll}
\Tau_1 & = \left\{ \nullsoftset, \absolutesoftset[X], (F_1,\E), (F_2,\E) \right\}, \\
\Tau_2 & = \left\{ \nullsoftset, \absolutesoftset[X], (F_3,\E) \right\}, \\
\Tau_3 & = \left\{ \nullsoftset, \absolutesoftset[X], (F_4,\E), (F_5,\E), (F_6,\E) \right\}, \\
\Tau_4 & = \left\{ \nullsoftset, \absolutesoftset[X], (F_7,\E), (F_8,\E) \right\}
\end{array}
$$
where the soft sets $(F_i,\E)$ (with $i=1,\ldots 8$) over $X$ are respectively defined by setting:
$$
\begin{array}{lll}
F_1(e_1)= \{ h_1 \}, \qquad & F_1(e_2) = \{ h_2, h_4 \}, \qquad & F_1(e_3) = \{ h_3 \}, \\
F_2(e_1)= \{ h_1, h_2 \},  & F_2(e_2) = \{ h_2, h_4, h_6 \},  & F_2(e_3) = \{ h_2, h_3 \}, \\[2mm]
F_3(e_1)= \{ h_3 \},  & F_3(e_2) = \{ h_4 \},  & F_3(e_3) = \{ h_5 \}, \\[2mm]
F_4(e_1)= \{ h_4 \},  & F_4(e_2) = \{ h_4 \},  & F_4(e_3) = \{ h_6 \}, \\
F_5(e_1)= \{ h_4, h_5 \},  & F_5(e_2) = \{ h_4, h_6 \},  & F_5(e_3) = \{ h_6, h_8 \}, \\
F_6(e_1)= X,  & F_6(e_2) = \{ h_4, h_6 \},  & F_6(e_3) = \{ h_5, h_6, h_8 \}, \\[2mm]
F_7(e_1)= \{ h_5, h_7 \},  & F_7(e_2) = \{ h_7 \},  & F_7(e_3) = \{ h_6, h_8 \}, \\
F_8(e_1)= \{ h_5, h_7, h_8 \},  & F_8(e_2) = X,  & F_8(e_3) = \{ h_6, h_7, h_8 \}.
\end{array}
$$
Then $\left( X, \Tau_i, 4, \E \right)$ is a soft $4$-topological space over $X$.
Let us note that $\bigcup_{i=1}^4 \Tau_i$ is not a soft topology since, for example,
$(F_4, \E) \softcup (F_7,\E)$ is the soft set $(H,\E)$ defined by
$H(e_1) = \{ h_4, h_5, h_7\}, H(e_2) = \{ h_4, h_7 \},  H(e_3) = \{ h_6, h_8 \}$
and $(H,\E)\notin \bigcup_{i=1}^4 \Tau_i$.
\end{example}

From Definitions \ref{def:softntopologicalspace} and \ref{def:softnopenset} immediately
derives the following proposition.

\begin{proposition}
\label{pro:relationshiptosoftntopology}
Let $\left( X, \Tau_i, N, \E \right)$ be a soft $N$-topological space over $X$, then we have that:
\begin{enumerate}
\item every $\Tau_i$ soft open set (with $i=1,\ldots N$) is a soft $N$-open set
\item every $\Tau_i$ soft closed set (with $i=1,\ldots N$) is a soft $N$-closed set
\end{enumerate}
\end{proposition}

Let us note that the converse of the above statements does not hold in general
since can exists some soft $N$-open (closed) set which is not a $\Tau_i$ soft open (closed) set
for every $i=1,\ldots N$.


\begin{proposition}
\label{pro:softNtopologicalspaceToNtopologicalSpaces}
Let $\left( X, \Tau_i, N, \E \right)$ be a soft $N$-topological space over $X$.
Then, for every parameter $e \in \E$,
the structure $(X, {\Tau_1}_e, \ldots {\Tau_N}_e)$ where
${\Tau_i}_e = \left\{ F(e) : \, (F,\E) \in \Tau_i \right\}$ (with $i=1,\ldots N$)
is an $N$-topological space.
\end{proposition}
\begin{proof}
In fact, for every $i=1,\ldots N$ and every paramter $e \in \E$,
by Proposition \ref{pro:softtopologytotopologies},
any ${\Tau_i}_e$ is a crisp topology on $X$ and hence, by Definition \ref{def:softntopologicalspace},
$(X, {\Tau_1}_e, \ldots {\Tau_N}_e)$ is a (crisp) $N$-topological space on $X$
in the sense of definition given in \cite{tawfiq}.
\end{proof}

In other words, Proposition \ref{pro:softNtopologicalspaceToNtopologicalSpaces} states
that every soft $N$-topological space over $X$ gives a parameterized family of
(crisp) $N$-topological space over the same set $X$.

\begin{example}
\label{ex:softNtopologicalspaceToNtopologicalSpaces}
If we consider the soft $4$-topological space $\left( X, \Tau_i, 4, \E \right)$ over $X$
defined in the Example \ref{ex:softntopologicalspace},
then, the parameterized families of crisp topologies respect to each parameter of $\E$ are:
$$
\begin{array}{ll}
{\Tau_1}_{e_1} & = \left\{ \emptyset, X, \{h_1\}, \{h_1, h_2\} \right\}, \\
{\Tau_1}_{e_2} & = \left\{ \emptyset, X, \{h_2, h_4\}, \{h_2, h_4, h_6\} \right\}, \\
{\Tau_1}_{e_3} & = \left\{ \emptyset, X, \{h_3\}, \{h_2, h_3\}  \right\}, \\[2mm]
{\Tau_2}_{e_1} & = \left\{ \emptyset, X, \{h_3\} \right\}, \\
{\Tau_2}_{e_2} & = \left\{ \emptyset, X, \{h_4\} \right\}, \\
{\Tau_2}_{e_3} & = \left\{ \emptyset, X, \{h_5\} \right\}, \\[2mm]
{\Tau_3}_{e_1} & = \left\{ \emptyset, X, \{h_4\}, \{h_4, h_5\} \right\}, \\
{\Tau_3}_{e_2} & = \left\{ \emptyset, X, \{h_4\}, \{h_4, h_6\}  \right\}, \\
{\Tau_3}_{e_3} & = \left\{ \emptyset, X, \{h_6\}, \{h_6, h_8\}, \{h_5, h_6, h_8\}  \right\}, \\[2mm]
{\Tau_4}_{e_1} & = \left\{ \emptyset, X, \{h_5, h_7\}, \{h_5, h_7, h_8\}  \right\}, \\
{\Tau_4}_{e_2} & = \left\{ \emptyset, X, \{h_7\} \right\}, \\
{\Tau_4}_{e_3} & = \left\{ \emptyset, X, \{h_6, h_8\}, \{h_6, h_7, h_8\} \right\}  
\end{array}
$$
and so $(X, {\Tau_1}_{e_j}, {\Tau_2}_{e_j}, {\Tau_3}_{e_j}, {\Tau_4}_{e_j})$ (with $j=1,\ldots 3$)
are crisp $N$-topological spaces on $X$
in the sense of definition given by Tawfiq and Majeed in \cite{tawfiq}.
\end{example}

Recalling the Definition \ref{def:softrelativetopology}, we can give the following proposition.

\begin{proposition}
\label{pro:relative_soft_n-topological_space}
Let $\left( X, \Tau_i, N, \E \right)$ be a soft $N$-topological space over $X$
and $Y$ be a nonempty subset of $X$, then
$\left( Y, {\Tau_i}_Y, N, \E \right)$, i.e.
the set $Y$ equipped with the relative soft topologies of ${\Tau_i}_Y$ on $Y$
(with $i=1,\ldots N$) is a soft $N$-topological space over $Y$.
\end{proposition}
\begin{proof}
In fact, for every $i=1,\ldots N$, by Proposition \ref{pro:softrelativetopology},
any ${\Tau_i}_Y$ is a soft topology on $Y$ and hence, by Definition \ref{def:softntopologicalspace},
$\left( Y, {\Tau_i}_Y, N, \E \right)$ is a soft $N$-topological space over $Y$.
\end{proof}

\begin{definition}
\label{def:relative_soft_n-topological_space}
Let $\left( X, \Tau_i, N, \E \right)$ be a soft $N$-topological space over $X$
and $Y$ be a nonempty subset of $X$
the soft $N$-topological space $\left( Y, {\Tau_i}_Y, N, \E \right)$
is said to be the \df{relative soft $N$-topological space} of $\left( X, \Tau_i, N, \E \right)$ on $Y$
or the \df{soft $N$-topological subspace} of $\left( X, \Tau_i, N, \E \right)$ on $Y$.
\end{definition}

\begin{example}
\label{ex:softntopologicalsubspace}
Let us consider the soft $4$-topological space $\left( X, \Tau_i, 4, \E \right)$ over $X$
defined in the Example \ref{ex:softntopologicalspace} and the subset $Y = \{ h_1, h_3, h_4, h_5, h_8 \}$ of $X$.
Then the sub soft sets $(\rl F_i,\E)$ (with $i=1,\ldots 8$) of the soft open set $(F_i,\E)$ over $Y$
results to be defined by:
$$
\begin{array}{lll}
\rl F_1(e_1)= \{ h_1 \}, \qquad & \rl F_1(e_2) = \{ h_4 \}, \qquad & \rl F_1(e_3) = \{ h_3 \}, \\
\rl F_2(e_1)= \{ h_1 \},  & \rl F_2(e_2) = \{ h_4 \},  & \rl F_2(e_3) = \{ h_3 \}, \\[2mm]
\rl F_3(e_1)= \{ h_3 \},  & \rl F_3(e_2) = \{ h_4 \},  & \rl F_3(e_3) = \{ h_5 \}, \\[2mm]
\rl F_4(e_1)= \{ h_4 \},  & \rl F_4(e_2) = \{ h_4 \},  & \rl F_4(e_3) = \emptyset, \\
\rl F_5(e_1)= \{ h_4, h_5 \},  & \rl F_5(e_2) = \{ h_4 \},  & \rl F_5(e_3) = \{  h_8 \}, \\
\rl F_6(e_1)= Y,  & \rl F_6(e_2) = \{ h_4 \},  & \rl F_6(e_3) = \{ h_5, h_8 \}, \\[2mm]
\rl F_7(e_1)= \{ h_5 \},  & \rl F_7(e_2) = \emptyset,  & \rl F_7(e_3) = \{ h_8 \}, \\
\rl F_8(e_1)= \{ h_5, h_8 \},  & \rl F_8(e_2) = Y,  & \rl F_8(e_3) = \{ h_8 \}
\end{array}
$$
and so, being $(\rl F_1,\E) \softequal (\rl F_2,\E)$,
the soft $4$-topological subspace $\left( Y, {\Tau_i}_Y, 4, \E \right)$
of $\left( X, \Tau_i, 4, \E \right)$ on $Y$
is formed by the following soft relative topologies ${\Tau_i}_Y$ of $\Tau_i$ on $Y$ (with $i=1,\ldots 4$):
$$
\begin{array}{ll}
{\Tau_1}_Y = \left\{ \nullsoftset, \absolutesoftset[Y], (\rl F_1,\E) \right\} , \\
{\Tau_2}_Y= \left\{ \nullsoftset, \absolutesoftset[Y], (\rl F_3,\E) \right\} , \\
{\Tau_3}_Y = \left\{ \nullsoftset, \absolutesoftset[Y], (\rl F_4,\E), (\rl F_5,\E), (\rl F_6,\E) \right\} , \\
{\Tau_4}_Y = \left\{ \nullsoftset, \absolutesoftset[Y], (\rl F_7,\E), (\rl F_8,\E) \right\} .
\end{array}
$$
\end{example}


\begin{definition}
\label{def:N-wisesoftT0}
A soft $N$-topological space $\left( X, \Tau_i, N, \E \right)$ over $X$
is called an \df{$N$-wise soft $T_0$-space}
if for every pair of distinct soft points $(x_\alpha, \E), (y_\beta, \E) \in \SPE[X]$
there exists a soft $N$-open set which soft contains exactly one of these points, i.e.
there is some soft $N$-open set $(A,\E)$
such that $(x_\alpha, \E) \softin (A,\E)$ and $(y_\beta, \E) \softnotin (A,\E)$,
or there is some soft $N$-open set $(B,\E)$
such that $(y_\beta, \E) \softin (B,\E)$ and $(x_\alpha, \E) \softnotin (B,\E)$.
\end{definition}

\begin{proposition}
\label{pro:softT0_N-wisesoftT0}
Let $\left( X, \Tau_i, N, \E \right)$ be a soft $N$-topological space over $X$.
If at least one soft topological space $(X,\Tau_j, \E)$ (for some $j = 1, \ldots N$) is a soft $T_0$-space,
then $\left( X, \Tau_i, N, \E \right)$ is an $N$-wise soft $T_0$-space.
\end{proposition}
\begin{proof}
Suppose that there exists some $j = 1, \ldots N$ such that $(X,\Tau_j, \E)$ is a soft $T_0$-space.
Then, for every pair of distinct soft points $(x_\alpha, \E), (y_\beta, \E) \in \SPE[X]$,
by Definition \ref{def:softT0}, there is some $(A,\E) \in \Tau_j$ such that $(x_\alpha, \E) \softin (A,\E)$
and $(y_\beta, \E) \softnotin (A,\E)$,
or there is some $(B,\E) \in \Tau_j$ such that $(y_\beta, \E) \softin (B,\E)$
and $(x_\alpha, \E) \softnotin (B,\E)$.
Since, by Proposition \ref{pro:relationshiptosoftntopology}, every $\Tau_j$ soft open set
is a soft $N$-open set,
it follows that $(A,\E)$ and $(B,\E)$ are soft $N$-open sets
and so that Definition \ref{def:N-wisesoftT0} holds,
i.e. that $\left( X, \Tau_i, N, \E \right)$ is an $N$-wise soft $T_0$-space.
\end{proof}

\begin{remark}
\label{rem:converse_softT0_N-wisesoftT0}
The converse of Proposition \ref{pro:softT0_N-wisesoftT0} is not true in general,
that is can exist an $N$-wise soft $T_0$-space which soft topologies are not soft $T_0$
as it is shown in the following counterexample.
\end{remark}

\begin{example}
\label{ex:converse_softT0_N-wisesoftT0}
Let $X= \{ h_1, h_2, h_3 \}$ be the universe set and $\E = \{ e_1, e_2 \}$ be the set of parameters.
Consider the following $N=2$ soft topologies over $X$:
$$
\begin{array}{ll}
\Tau_1 & = \left\{ \nullsoftset, \absolutesoftset[X], (F_1,\E) \right\}, \\
\Tau_2 & = \left\{ \nullsoftset, \absolutesoftset[X], (F_2,\E), (F_3,\E), (F_4,\E), (F_5,\E) \right\}, \\
\end{array}
$$
where the soft open sets $(F_i,\E)$ (with $i=1,\ldots 4$) over $X$ are respectively defined by setting:
$$
\begin{array}{ll}
F_1(e_1)= \{ h_1 \}, \qquad & F_1(e_2) = \{ h_2 \}, \\[2mm]
F_2(e_1)= \{ h_1 \},  & F_2(e_2) = \{ h_3 \}, \\
F_3(e_1)= \{ h_1, h_3 \},  & F_3(e_2) = \{ h_1, h_3 \}, \\
F_4(e_1)= \{ h_2, h_3 \},  & F_3(e_2) = \{ h_2, h_3 \}, \\
F_5(e_1)= \absolutesoftset[X],  & F_4(e_2) = \{ h_2, h_3 \} .
\end{array}
$$
and hence the soft $2$-topological space $\left( X, \Tau_i, 2, \E \right)$.

After noticing that the set $\SPE[X]$ of all soft points over $X$ contains the following $6$ members:
$$
\SPE[X] = \left\{
\left( {h_1}_{e_1} , \E \right), \left( {h_1}_{e_2} , \E \right),
\left( {h_2}_{e_1} , \E \right), \left( {h_2}_{e_2} , \E \right),
\left( {h_3}_{e_1} , \E \right), \left( {h_3}_{e_2} , \E \right)
\right\}
$$
we can easily verify that, for every pair of distinct soft points
(of the $\frac{6(6-1)}{2}=15$ possible combinations),
there exists a soft $N$-open set which soft contains exactly one of these points
i.e., for example, that:
\begin{itemize}
\item $(F_1, \E)$ soft contains $\left( {h_1}_{e_1} , \E \right)$ but not $\left( {h_1}_{e_2} , \E \right)$,
\item $(F_1, \E)$ soft contains $\left( {h_1}_{e_1} , \E \right)$ but not $\left( {h_2}_{e_1} , \E \right)$,
\item $(F_3, \E)$ soft contains $\left( {h_1}_{e_1} , \E \right)$ but not $\left( {h_2}_{e_2} , \E \right)$,
\item \ldots\ldots\ldots
\item $(F_4, \E)$ soft contains $\left( {h_3}_{e_1} , \E \right)$ but not $\left( {h_1}_{e_2} , \E \right)$,
\item $(F_2, \E)$ soft contains $\left( {h_3}_{e_2} , \E \right)$ but not $\left( {h_2}_{e_1} , \E \right)$,
\item etc.
\end{itemize}
and so that $\left( X, \Tau_i, 2, \E \right)$ is a $2$-wise soft $T_0$-space.
However, neither the soft topological spaces $(X,\Tau_1,\E)$ and $(X,\Tau_2,\E)$ are soft $T_0$.
In fact, $(X,\Tau_1,\E)$ is a not a soft $T_0$-space
since for $x= \left( {h_1}_{e_1} , \E \right)$ and $y= \left( {h_2}_{e_2} , \E \right)$,
in $\Tau_1$, there is no soft open set that soft contains $x$ but not $y$
and no soft open set that soft contains $y$ but not $x$.
Similarly, $(X,\Tau_2,\E)$ is a not a soft $T_0$-space
since for $x= \left( {h_2}_{e_1} , \E \right)$ and $y= \left( {h_2}_{e_2} , \E \right)$,
in $\Tau_2$, there is no soft open setthat soft contains $x$ but not $y$
and no soft open set that soft contains $y$ but not $x$.
\end{example}


\begin{proposition}
\label{pro:N-wisesoftT0_supsoftT0}
If the soft $N$-topological space $\left( X, \Tau_i, N, \E \right)$ over $X$
is an $N$-wise soft $T_0$-space, then the supremum soft topological space
$\left( X, \bigvee_{i=1}^N \Tau_i , \E \right)$ is a soft $T_0$-space.
\end{proposition}
\begin{proof}
Suppose that $\left( X, \Tau_i, N, \E \right)$ is an $N$-wise soft $T_0$-space.
Then, for every pair of distinct soft points $(x_\alpha, \E), (y_\beta, \E) \in \SPE[X]$,
by Definition \ref{def:N-wisesoftT0} and Remark \ref{rem:alternativedefinitionofsoftnopenset}
there exists some soft $N$-open set $(A,\E) \in \bigcup_{i=1}^N \Tau_i$
such that $(x_\alpha, \E) \softin (A,\E)$ and $(y_\beta, \E) \softnotin (A,\E)$,
or there exists some soft $N$-open set $(B,\E) \in \bigcup_{i=1}^N \Tau_i$
such that $(y_\beta, \E) \softin (B,\E)$ and $(x_\alpha, \E) \softnotin (B,\E)$.
Since, by Definition \ref{def:supsofttopologies},
the supremum soft topology $\bigvee_{i=1}^N \Tau_i$ of all soft topologies $\Tau_i$ (with $i=1,\ldots N$)
is the smallest soft topology on $X$ containing the union $\bigcup_{i=1}^N \Tau_i$,
we have that both $(A,\E)$ and $(B,\E)$ also belong to $\bigvee_{i=1}^N \Tau_i$ and hence
that the supremum soft topological space $\left( X, \bigvee_{i=1}^N \Tau_i , \E \right)$ is a soft $T_0$-space.
\end{proof}

\begin{remark}
\label{rem:converse_N-wisesoftT0_supsoftT0}
The converse of Proposition \ref{pro:softT0_N-wisesoftT0} is not true in general,
that is can exist an $N$-wise soft $T_0$-space
such that the supremum soft topological space of its soft topologies is not a soft $T_0$-space,
as can be seen in the following counterexample.
\end{remark}

\begin{example}
\label{ex:converse_N-wisesoftT0_supsoftT0}
Let $X= \{ h_1, h_2, h_3 \}$ be the universe set, $\E = \{ e_1, e_2 \}$ be the set of parameters
and consider the soft $2$-topological space $\left( X, \Tau_i, 2, \E \right)$
defined in the Example \ref{ex:converse_softT0_N-wisesoftT0}.
The supremum soft topology $\Tau_1 \vee \Tau_2$, being the smallest soft topology over $X$ containing the
crisp union $\Tau_1 \cup \Tau_2$, results to be:
$$\Tau_1 \vee \Tau_2
 = \left\{ \nullsoftset, \absolutesoftset[X], (F_1,\E), (F_2,\E), (F_3,\E), (F_4,\E),
  (F_5,\E), (F_6,\E), (F_7,\E), (F_8,\E), (F_9,\E) \right\}$$
where the new soft open sets $(F_i,\E)$ (with $i=6,\ldots 9$) are:
$$
\begin{array}{llll}
(F_6,\E) = (F_1,\E) \softcup (F_2,\E) \hspace{5mm} & \textit{defined by }\hspace{4mm}
   & F_6(e_1)= \{ h_1 \}, \quad & F_6(e_2) = \{ h_2, h_3 \}, \\
(F_7,\E) = (F_1,\E) \softcup (F_3,\E)  & \textit{defined by } &  F_7(e_1)= \{ h_1, h_3 \}, & F_7(e_2) = X, \\
(F_8,\E) = (F_1,\E) \softcap (F_2,\E)  & \textit{defined by } &  F_8(e_1)= \{ h_1 \}, & F_8(e_2) = \emptyset, \\
(F_9,\E) = (F_1,\E) \softcap (F_3,\E)  & \textit{defined by } &  F_9(e_1)= \emptyset, & F_9(e_2) = \{ h_2 \} .
\end{array}
$$
However, the supremum soft topological space $\left( X, \Tau_1 \vee \Tau_2 , \E \right)$
is not a soft $T_0$-space,
since for $x= \left( {h_1}_{e_2} , \E \right)$ and $y= \left( {h_3}_{e_2} , \E \right)$,
in $\Tau_1 \vee \Tau_2$, there is no soft open set that soft contains $x$ but not $y$
and no soft open set that soft contains $y$ but not $x$.
\end{example}


Although the following result can be directly proven,
it is interesting to note that it can be achieved
from Proposition \ref{pro:softT0_N-wisesoftT0}
and Proposition \ref{pro:N-wisesoftT0_supsoftT0}.

\begin{corollary}
\label{cor:softT0_supsoftT0}
Let $\left( X, \Tau_i, N, \E \right)$ be a soft $N$-topological space over $X$.
If at least one soft topological space $(X,\Tau_j, \E)$ (for some $j = 1, \ldots N$) is a soft $T_0$-space,
then the supremum soft topological space
$\left( X, \bigvee_{i=1}^N \Tau_i , \E \right)$ is a soft $T_0$-space.
\end{corollary}


The $N$-wise soft $T_0$ property is hereditary. In fact, we have the following proposition.

\begin{proposition}
\label{pro:N-wisesoftT0ishereditary}
Let $\left( X, \Tau_i, N, \E \right)$ be an $N$-wise soft $T_0$-space over $X$
and $Y$ be a nonempty subset of $X$.
Then the soft $N$-topological subspace $\left( Y, {\Tau_i}_Y, N, \E \right)$
of $\left( X, \Tau_i, N, \E \right)$ on $Y$ is an $N$-wise soft $T_0$-space.
\end{proposition}
\begin{proof}
Suppose that $\left( X, \Tau_i, N, \E \right)$ is an $N$-wise soft $T_0$-space
and let $Y$ be a nonempty subset of $X$.
Then, for every pair of distinct soft points $(x_\alpha, \E), (y_\beta, \E) \in \SPE[Y]$,
it follows, in particular, that $(x_\alpha, \E), (y_\beta, \E)$ are soft points in $X$ too.
So, by hypothesis, there exists some soft $N$-open set $(A,\E)$
such that $(x_\alpha, \E) \softin (A,\E)$ and $(y_\beta, \E) \softnotin (A,\E)$,
or there exists some soft $N$-open set $(B,\E)$
such that $(y_\beta, \E) \softin (B,\E)$ and $(x_\alpha, \E) \softnotin (B,\E)$.
Thus, by Definition \ref{def:softnopenset}, there exists some $j,k \in \{1,\ldots N\}$
such that $(A,\E) \in \Tau_j$ and $(B,\E) \in \Tau_k$.
Hence, by Remark \ref{rem:subsoftset} and Definition \ref{def:softrelativetopology}, it follows that
the soft intersections $(A,\E) \softcap (\tilde{Y}, \E) = (\rl[Y] A,\E)$
and $(B,\E) \softcap (\tilde{Y}, \E) = (\rl[Y] B,\E)$
belong to the soft relative topologies $(Y, {\Tau_j}_Y, \E)$ and $(Y, {\Tau_k}_Y, \E)$ respectively
and they are such that $(x_\alpha, \E) \softin (\rl[Y] A,\E)$ and $(y_\beta, \E) \softnotin (\rl[Y] A,\E)$,
or such that $(y_\beta, \E) \softin (\rl[Y] B,\E)$ and $(x_\alpha, \E) \softnotin (\rl[Y] B,\E)$.
Finally, since by Definition \ref{def:relative_soft_n-topological_space}, we have that
both $(\rl[Y] A,\E)$ and $(\rl[Y] B,\E)$
belong to the relative soft $N$-topological space $\left( Y, {\Tau_i}_Y, N, \E \right)$,
we have that Definition \ref{def:N-wisesoftT0} holds, and so that
the soft $N$-topological subspace $\left( Y, {\Tau_i}_Y, N, \E \right)$ is an $N$-wise soft $T_0$-space.
\end{proof}

\begin{remark}
\label{rem:N-wisesoftT0subspaceofaNON-wisesoftT0}
The converse of Proposition \ref{pro:N-wisesoftT0ishereditary} is false,
i.e. can exist a non $N$-wise soft $T_0$-space $\left( X, \Tau_i, N, \E \right)$
which, for some $Y\subset X$, has a $N$-wise soft $T_0$-subspace $\left( Y, {\Tau_i}_Y, N, \E \right)$
as it is shown in the following counterexample.
\end{remark}

\begin{example}
\label{ex:N-wisesoftT0subspaceofaNON-wisesoftT0}
Let $X= \{ h_1, h_2, h_3 \}$ be the universe set and
$\E = \{ e_1, e_2 \}$ be the set of parameters.
Consider the soft $2$-topological space $\left( X, \Tau_i, 2, \E \right)$
built by the following $N=2$ soft topologies over $X$:
$$
\begin{array}{ll}
\Tau_1 & = \left\{ \nullsoftset, \absolutesoftset[X], (F_1,\E) \right\}, \\
\Tau_2 & = \left\{ \nullsoftset, \absolutesoftset[X], (F_2,\E), (F_3,\E) \right\}
\end{array}
$$
where the soft sets $(F_i,\E)$ (with $i=1,\ldots 3$) over $X$ are respectively defined by setting:
$$
\begin{array}{ll}
F_1(e_1)= \{ h_1 \}, \qquad & F_1(e_2) = \{ h_2 \}, \\[2mm]
F_2(e_1)= \{ h_2 \},  & F_2(e_2) = \{ h_1 \}, \\
F_3(e_1)= \{ h_2, h_3 \},  & F_3(e_2) = \{ h_1, h_2 \} .
\end{array}
$$
Then $\left( X, \Tau_i, 2, \E \right)$ is not an $N$-wise soft $T_0$-space,
since for $x= \left( {h_2}_{e_1} , \E \right)$ and $y= \left( {h_1}_{e_2} , \E \right)$,
there is no soft $N$-open set that soft contains $x$ but not $y$
and no soft $N$-open set that soft contains $y$ but not $x$.

However, if we consider the subset $Y = \{ h_1, h_3 \}$ of $X$,
the sub soft sets $(\rl F_i,\E)$ (with $i=1,\ldots 3$) of the soft open set $(F_i,\E)$ over $Y$
results to be defined by:
$$
\begin{array}{lll}
\rl F_1(e_1) = \{ h_1 \} , \qquad & \rl F_1(e_2) = \emptyset ,\\[2mm]
\rl F_2(e_1) = \emptyset ,  & \rl F_2(e_2) = \{ h_1 \} , \\
\rl F_3(e_1) = \{ h_3 \} ,  & \rl F_3(e_2) = \{ h_1 \} .
\end{array}
$$
Thus, the soft $2$-topological subspace $\left( Y, {\Tau_i}_Y, 2, \E \right)$
of $\left( X, \Tau_i, 2, \E \right)$ on $Y$
is formed by the following soft relative topologies ${\Tau_i}_Y$ of $\Tau_i$ on $Y$ (with $i=1,\ldots 2$):
$$
\begin{array}{ll}
{\Tau_1}_Y = \left\{ \nullsoftset, \absolutesoftset[Y], (\rl F_1,\E) \right\} \\
{\Tau_2}_Y= \left\{ \nullsoftset, \absolutesoftset[Y], (\rl F_2,\E), (\rl F_3,\E) \right\}
\end{array}
$$
and it is trivially checked that it is an $N$-wise soft $T_0$-space.
\end{example}


\begin{definition}
\label{def:N-wisesoftT1}
A soft $N$-topological space $\left( X, \Tau_i, N, \E \right)$ over $X$
is called an \df{$N$-wise soft $T_1$-space}
if for every pair of distinct soft points $(x_\alpha, \E), (y_\beta, \E) \in \SPE[X]$
there exists a soft $N$-open set $(A,\E)$ which soft contains one soft point but not the other one, that is
$(x_\alpha, \E) \softin (A,\E)$ and $(y_\beta, \E) \softnotin (A,\E)$.
\end{definition}

\begin{definition}
\label{def:N-wisesoftT2}
A soft $N$-topological space $\left( X, \Tau_i, N, \E \right)$ over $X$
is called an \df{$N$-wise soft $T_2$-space} or \df{$N$-wise soft Hausdorff space}
if for every pair of distinct soft points $(x_\alpha, \E), (y_\beta, \E) \in \SPE[X]$
there exist two soft $N$-open sets $(A,\E), (B,\E)$ which are soft disjoints
and soft contain the two soft points respectively, that is
$(x_\alpha, \E) \softin (A,\E)$, $(y_\beta, \E) \softin (B,\E)$
and $(A,\E) \softcap (B,\E) = \nullsoftset$.
\end{definition}


\begin{proposition}
\label{pro:N-wisesoftT2impliesN-wisesoftT1impliesN-wisesoftT0}
Every $N$-wise soft $T_2$-space is an $N$-wise soft $T_1$-space
and every $N$-wise soft $T_1$-space is an $N$-wise soft $T_0$-space
\end{proposition}
\begin{proof}
It easily follows from Definitions \ref{def:N-wisesoftT0}, \ref{def:N-wisesoftT1} and \ref{def:N-wisesoftT2}.
\end{proof}


\begin{proposition}
\label{pro:softT1T2_N-wisesoftT1T2}
Let $\left( X, \Tau_i, N, \E \right)$ be a soft $N$-topological space over $X$.
If at least one soft topological space $(X,\Tau_j, \E)$ (for some $j = 1, \ldots N$)
is a soft $T_1$-space (respectively a soft $T_2$-space),
then $\left( X, \Tau_i, N, \E \right)$ is an $N$-wise soft $T_1$-space
(resp. an $N$-wise soft $T_2$-space).
\end{proposition}
\begin{proof}
Similar to the proof of Proposition \ref{pro:softT0_N-wisesoftT0}.
\end{proof}

\begin{proposition}
\label{pro:N-wisesoftT1T2_supsoftT1T2}
If the soft $N$-topological space $\left( X, \Tau_i, N, \E \right)$ over $X$
is an $N$-wise soft $T_1$-space (respectively an $N$-wise soft $T_2$-space),
then the supremum soft topological space
$\left( X, \bigvee_{i=1}^N \Tau_i , \E \right)$
is a soft $T_1$-space (resp. a soft $T_2$-space).
\end{proposition}
\begin{proof}
Similar to the proof of Proposition \ref{pro:N-wisesoftT0_supsoftT0}.
\end{proof}

\begin{proposition}
\label{pro:N-wisesoftT1T2ishereditary}
Let $\left( X, \Tau_i, N, \E \right)$ be an $N$-wise soft $T_1$-space
(respectively an $N$-wise soft $T_2$-space) over $X$
and $Y$ be a nonempty subset of $X$.
Then the soft $N$-topological subspace $\left( Y, {\Tau_i}_Y, N, \E \right)$ on $Y$
is an $N$-wise soft $T_1$-space (resp. an $N$-wise soft $T_2$-space).
\end{proposition}
\begin{proof}
Similar to the proof of Proposition \ref{pro:N-wisesoftT0ishereditary}.
\end{proof}

\section{Conclusion}
We have defined the new notion of soft $N$-topological space
as generalization both of the concepts of soft topological space given by Shabir and Naz \cite{shabir}
and that of $N$-topological space as introduced by Tawfiq and Majeed \cite{tawfiq}
and, independently, by Khan \cite{khan}
and we have investigated some basic properties of such class of spaces with particular regard to its subspaces
and to the parameterized families of crisp topologies generated by the soft $N$-space.
We have also introduced and studied some new separation axioms called $N$-wise soft $T_0$,
$N$-wise soft $T_1$, and $N$-wise soft $T_2$.

This paper is just a beginning of the investigation of a new kind of structure.
So, it will be necessary to continue the study and carry out more theoretical research
in order to build a general framework for practical applications.


\end{document}